\documentclass[12pt]{amsart}
\usepackage{amscd,amsmath,amssymb,enumerate,amsfonts,mathrsfs,txfonts}
\usepackage{comment}
\usepackage{hyperref}
\usepackage{xcolor}
\usepackage{tikz,pgfplots,pgf}
\usepackage[all]{xy}

\tikzset{node distance=3cm, auto}
\usetikzlibrary{calc} 
\usetikzlibrary{trees} 

\usepackage{vmargin}
\setpapersize{A4}
\setmargins{2.5cm}      
{1.5cm}                        
{16.5cm}                      
{23.42cm}                   
{10pt}                          
{1cm}                           
{0pt}                          
{2cm}
\newtheorem{theorem}{Theorem}[section]
\newtheorem{lemma}[theorem]{Lemma}
\newtheorem{proposition}[theorem]{Proposition}

\theoremstyle{definition}
\newtheorem{definition}[theorem]{Definition}
\newtheorem{example}[theorem]{Example}

\newtheorem{question}[theorem]{Question}

\newtheorem{remark}[theorem]{Remark}
\newtheorem{remarks}[theorem]{Remarks}
\numberwithin{equation}{section}

\newcommand{\pten}{\ensuremath{\widehat{\otimes}_\pi}}

\def\N{\mathbb{N}}
\def\R{\mathbb{R}}

\def\F{\mathcal{F}}

\def\L{\mathcal{L}}

\def\P{\mathcal{P}}

\def\Id{\mathrm{Id}}
\def\Lipo{\mathrm{Lip}_0}
\def\lipo{\mathrm{lip}_0}
\def\Lip{\mathrm{Lip}}
\def\lip{\mathrm{lip}}

\def\sign{\mathrm{sign}}

\makeatletter
\@namedef{subjclassname@2020}{\textup{2020} Mathematics Subject Classification}
\makeatother

\begin{document}

\title[Lipschitz interpolating sequences]{Lipschitz interpolating sequences}

\author[A. Jim{\'e}nez-Vargas]{A. Jim{\'e}nez-Vargas}
\address[A. Jim{\'e}nez-Vargas]{Departamento de Matem\'aticas, Universidad de Almer\'ia, Ctra. de Sacramento s/n, 04120 La Ca\~{n}ada de San Urbano, Almer\'ia, Spain.}
\email{ajimenez@ual.es}
\urladdr{\url{https://brujula.ual.es/authors/485.html}}

\author[A. Rueda Zoca]{Abraham Rueda Zoca}
\address[A. Rueda Zoca]{Universidad de Granada, Facultad de Ciencias. Departamento de An\'{a}lisis Matem\'{a}tico, 18071-Granada(Spain)} 
\email{abrahamrueda@ugr.es}
\urladdr{\url{https://arzenglish.wordpress.com}}

\date{\today}




\begin{abstract} 
Let $X$ be a metric space with a base point $0$, and let $\Lipo(X)$ be the Banach space of all Lipschitz functions $f\colon X\to\R$ such that $f(0)=0$. Given a set of points $\left((x_i,y_i)\right)_{i\in I}$ in $X^2$ with $x_i\neq y_i$ for all $i\in I$, we study the following interpolation problem: when for each bounded set $\left(\alpha_i\right)_{i\in I}$ in $\mathbb{R}$ the algorithm  
$$
\frac{f(x_i)-f(y_i)}{d(x_i,y_i)}=\alpha_i\qquad (i\in I)
$$
can be implemented by a function $f\in\Lipo(X)$? Our approach involves the concept of a Beurling set of functions in $\Lipo(X)$ for $\left((x_i,y_i)\right)_{i\in I}$ which has shown to be useful in the so-called transportation problem.
\end{abstract}
\maketitle


\section{Introduction}\label{section 0}

Given a nonempty set $X$, let $A(X)$ be a space of scalar-valued functions defined on $X$. A sequence $\left(x_n\right)_{n\in\N}$ of elements in $X$ is said to be an interpolating sequence for $A(X)$ if for each bounded scalar-valued sequence $\left(\alpha_n\right)_{n\in\N}$, there exists a function $f\in A(X)$ such that $f(x_n)=\alpha_n$ for all $n\in\N$.

The theory of interpolating sequences for spaces of analytic functions plays an important role in the theory of Banach spaces, not only because of the intrinsic beauty of the concept but also because of its applications in different topics such as operator theory, linear systems theory, control theory, function theory, Banach algebras theory and others (see the monographs of Garnett \cite{Gar-81} and Seip \cite{Sei-04}).

Since Carleson \cite{Car-58} characterized the interpolating sequences for the space of bounded holomorphic functions on the complex open unit disk, several authors have studied interpolating
sequences for different function spaces as, for example, spaces of Bloch functions \cite{Ate-92,MadMat-95,MirMal-22}, (weighted) spaces of analytic functions \cite{BlaGalLinMir-19,Muj-91} and uniform algebras \cite{GalGamLin-04,GalLinMir-09}, among others. 

Our aim in this paper is to raise and address a natural problem of interpolation for spaces of Lipschitz functions. Although the aforementioned results are naturally considered for complex-valued functions, we will focus on studying Lipschitz interpolating sets -- rather than sequences -- in the environment of real-valued Lipschitz functions. The main reason for considering only the real case is that our study requires to make use of the rich theory on Lipschitz-free spaces \cite{GodKal-03,Wea-18}, which has been mainly developed in the real setting since McShane extension theorem of Lipschitz functions is norm-preserving in the real-valued case but it is not in the complex-valued one (see \cite[Theorems 1.33 and Corollary 1.34]{Wea-18} for details).

Let $(X,d)$ be a pointed metric space with a base point denoted by $0$, let $E$ be a real Banach space and set $\widetilde{X}=\left\{(x,y)\in X\times X\colon x\neq y\right\}$. The pointed Lipschitz space $\Lipo(X,E)$ is the Banach space of all Lipschitz mappings $f\colon X\to E$ such that $f(0)=0$, endowed with the norm: 
$$
\Vert f\Vert:=\sup\left\{\frac{\left\|f(x)-f(y)\right\|}{d(x,y)}\colon (x,y)\in\widetilde{X}\right\}<\infty.
$$
The pointed little Lipschitz space $\lipo(X,E)$ is the closed subspace of $\Lipo(X,E)$ consisting of all those mappings $f\colon X\to E$ which satisfy the following property:
$$
\forall\varepsilon>0,\quad \exists\delta>0\colon x,y\in X,\, 0<d(x,y)<\delta\quad \Rightarrow\quad \frac{\left\|f(x)-f(y)\right\|}{d(x,y)}<\varepsilon.
$$
In the scalar-valued case, it is usual to write $\Lipo(X)$ and $\lipo(X)$ instead of $\Lipo(X,\R)$ and $\lipo(X,\R)$, respectively. 

For a nonempty set $I$, we denote by $\ell_1(I,E)$ and $\ell_\infty(I,E)$ the Banach spaces of absolutely summable families and bounded families of vectors in $E$, $\alpha=(\alpha_i)_{i\in I}$, under the respective norms: 
\begin{align*}
\left\|\alpha\right\|_1 &:=\sup\left\{\sum_{i\in F}\|\alpha_i\|\colon F\subseteq I, F \text{ finite}\right\},\\
\left\|\alpha\right\|_{\infty} &:=\sup\left\{\|\alpha_i\|\colon i\in I\right\}.
\end{align*}
In the case $E=\mathbb{R}$, we just write $\ell_1(I)$ and $\ell_\infty(I)$, and even $\ell_1$ and $\ell_\infty$ ($\ell^n_1$ and $\ell^n_\infty$) if in addition $I=\N$ (resp. $I=\{1,\ldots,n\})$. Given any $i\in I$, by $e_i$ we mean the element in $\ell_1(I)$ defined as $e_i(j)=\delta_{ij}$ for all $j\in I$, where $\delta_{ij}$ denotes the Kronecker delta of $i,j\in I$.

For real Banach spaces $E$ and $F$, we denote by $\L(E,F)$ the space of all bounded linear operators from $E$ into $F$, equipped with the operator canonical norm. In particular, we write $E^*$ instead of $\L(E,\mathbb{R})$. As usual, $B_E$ and $S_E$ stand for the closed unit ball and the unit sphere of $E$, respectively.

\begin{definition}\label{def-LIC-0}
Let $(X,d)$ be a pointed metric space and let $\left((x_i,y_i)\right)_{i\in I}$ be a set of points in $\widetilde{X}$. Clearly, the so-called Lipschitz interpolating operator associated to $\left((x_i,y_i)\right)_{i\in I}$, $T\colon\Lipo(X)\to\ell_\infty(I)$, given by 
$$
T(f):=\left(\frac{f(x_i)-f(y_i)}{d(x_i,y_i)}\right)_{i\in I}\qquad (f\in\Lipo(X)),
$$
is well-defined, linear and continuous with $0\neq ||T||\leq 1$. 

We will say that $\left((x_i,y_i)\right)_{i\in I}$ is a Lipschitz interpolating set in $\widetilde{X}$ for $\Lip_0(X)$ if for each element $\alpha=(\alpha_i)_{i\in I}\in\ell_\infty(I)$, the interpolation problem
$$
\frac{f(x_i)-f(y_i)}{d(x_i,y_i)}=\alpha_i\qquad (i\in I)
$$
can be solved with a function $f\in\Lip_0(X)$. This means that the operator $T\colon\Lipo(X)\to\ell_\infty$ is surjective and, equivalently, there exists a map $R\colon\ell_\infty\to\Lipo(X)$ such that $T\circ R=\Id_{\ell_\infty}$. In the case that $R\colon\ell_\infty\to\Lipo(X)$ is a bounded linear operator, we will say that $\left((x_i,y_i)\right)_{i\in I}$ is a linear Lipschitz interpolating set for $\Lip_0(X)$.
\end{definition}

If $\left((x_i,y_i)\right)_{i\in I}$ is a Lipschitz interpolating set in $\widetilde{X}$ for $\Lipo(X)$, then the operator $T\colon\Lipo(X)\to\ell_\infty(I)$ is surjective. An application of the open mapping theorem guarantees that there exists a constant $K\geq 1/\|T\|$ such that for each $\alpha\in B_{\ell_\infty(I)}$, we have $T(f)=\alpha$ for some $f\in\Lipo(X)$ with $\|f\|\leq K$. This justifies the introduction of the following constant. 

\begin{definition}\label{def-LIC}
Let $X$ be a pointed metric space and let $\left((x_i,y_i)\right)_{i\in I}$ be a Lipschitz interpolating set in $\widetilde{X}$ for $\Lipo(X)$. We define the Lipschitz interpolation constant for $\left((x_i,y_i)\right)_{i\in I}$ by
$$
M:=\inf\left\{K\geq 1 \;|\; \forall \alpha\in B_{\ell_\infty(I)}, \, \exists f\in K B_{\Lipo(X)} \colon T(f)=\alpha\right\}<\infty.
$$
\end{definition}

The Lipschitz interpolation constant for $\left((x_i,y_i)\right)_{i\in I}$ may admit different descriptions. For example, notice that $M=\|(\tilde{T})^{-1}\|$, where $\tilde{T}\colon \Lipo(X)/\ker(T)\to\ell_\infty(I)$ is the topological isomorphism given by $\tilde{T}(f+\ker(T))=T(f)$. On the other hand, if 
$$
M_\alpha:=\inf\left\{K\geq 1\,|\, \exists f\in KB_{\Lipo(X)} \colon T(f)=\alpha\right\} \qquad (\alpha\in B_{\ell_\infty(I)}),
$$
it is not hard to show that $M=\sup\left\{M_\alpha\colon \alpha\in B_{\ell_\infty(I)}\right\}$. Compare this last supremum to the interpolation constant introduced for bounded holomorphic functions of a sequence in the upper half plane (see \cite[p, 276]{Gar-81}), for elements of a uniform algebra of a sequence in its spectrum \cite{GalGamLin-04,GalLinMir-09}, for weighted analytic functions of a sequence in the unit ball of a Hilbert space \cite{BlaGalLinMir-19}, for (weighted) Bloch functions of a sequence in the complex open unit disk \cite{Ate-92,MirMal-22}, and for Lipschitz functions \cite{Kro-76}.

We now describe the content of this paper. After Section~\ref{section:prel}, where we will introduce some preliminary results, we will analyse in Section~\ref{section:Lipschitzseq} the problem of when a set $\left((x_i,y_i)\right)_{i\in I}$ in $\widetilde{X}$ is Lipschitz interpolating for $\Lipo(X)$. It turns out that the above is equivalent to the fact that the operator $S:\ell_1(I)\to\F(X)$ given by $S(e_i):=m_{x_i,y_i}$ is an into isomorphism (see Theorem~\ref{theo:caraLipschitintergeneral} for details and the role that the Lipschitz interpolation constant plays). In the search of examples of Lipschitz interpolating sets, we introduce in Definition~\ref{def:Berurlingeq} the notion of Beurling set of functions in $\Lipo(X)$ for $\left((x_i,y_i)\right)_{i\in I}$. Such sets of functions have shown its usefulness in the so-called transportation problem (see \cite{OstOst-20,OstOst-24}). Moreover, the existence of such a Beurling set of functions characterizes that the Lipschitz interpolation provides by the set $\left((x_i,y_i)\right)_{i\in I}$ becomes to be linear (see Theorems \ref{rem:levanvarapouloset} and \ref{cor:caravarapoulos0}). 

Among other characterizations (see (Theorem~\ref{cor:caravarapoulos})), we prove that a Lipschitz interpolating set $\left((x_i,y_i)\right)_{i\in I}$ in $\widetilde{X}$ for $\Lipo(X)$ admits a Beurling set of functions in $\Lipo(X)$ if, and only if, there exists an operator $P:\mathcal F(X)\to \ell_1(I)$ such that $P\circ S=\mathrm{Id}_{\ell_1(I)}$ and the norm of $P$ is exactly the same as the Lipschitz interpolation constant of $\left((x_i,y_i)\right)_{i\in I}$. The section also contains a number of examples of Lispchitz interpolating sets admitting a Beurling set of Lipschitz functions. Let us point out, among all the results, Theorem~\ref{theo:ovstroskii}, where we prove that if $\left((x_i,y_i)\right)_{i\in I}$ is a Lipschitz interpolating set whose associate constant is equal to $1$, then it admits a Beurling set of Lipschitz functions. This theorem, which is based on recent results by Ostrovskaa and Ostrovskii \cite{OstOst-24}, can be seen as a slight generalisation of \cite[Theorem 1.3]{OstOst-24}. At the end of the section, we study the separation and the stability of Lipschitz interpolating sets in $\widetilde{X}$ for $\Lipo(X)$ with respect to the so-called Lipschitz-molecular metric.

In Section~\ref{sect:compact}, we study Lipschitz interpolating sequences in the case of compact pointed metric spaces $X$. We prove in Proposition~\ref{prop:dist0interseqcompact} that if $((x_n,y_n))_{n=1}^\infty$ is a Lipschitz interpolating sequence in in $\widetilde{X}$ for $\Lipo(X)$, then $(d(x_n,y_n))_{n=1}^\infty\rightarrow 0$ as $n\to\infty$. Even though the converse is far from being true (Example~\ref{exam:compactnointer}), we prove the following kind of converse -- now without compactness on $X$ --: if $(d(x_n,y_n))_{n=1}^\infty\rightarrow 0$ as $n\to\infty$, then, for every $\varepsilon>0$, there exists a subsequence $((x_{n_k},y_{n_k}))_{k=1}^\infty$ which is Lipschitz interpolating for $\Lipo(X)$ and its Lipschitz interpolating constant is smaller than $1/(1-\varepsilon)$ (Theorem~\ref{theo:suceLorto}). These results are applied in Theorem~\ref{teo-1} to prove that if $\lipo(X)$ separates the points of $X$ uniformly and $((x_n,y_n))_{n=1}^\infty$ is a sequence in $\widetilde{X}$, then the associate Lipschitz interpolating operator $T:\Lipo(X)\to\ell_\infty$ is onto if, and only if, the operator $T_{|\lipo(X)}:\lipo(X)\to c_0$ is so, and both the Lipschitz interpolating constants of $T$ and $T_{|\lipo(X)}$ agree.

In our last Section~\ref{sect:vectorval}, we study Lipschitz interpolating operators from $\Lipo(X,E)$ to $\ell_\infty(I,E)$ for any Banach space $E$ (see Definition~\ref{def:Lipschitvect}). We prove in Theorem~\ref{theo:mainvectorval} that $\left((x_i,y_i)\right)_{i\in I}$ is a Lipschitz interpolating set in $\widetilde{X}$ for $\Lipo(X,E)$ for any Banach space $E$ with a common Lipschitz interpolating constant $M=M_E$ if, and only if, there exists a Beurling set of functions in $\Lipo(X)$ associated to $\left((x_i,y_i)\right)_{i\in I}$. In order to prove it we make use of tensor product theory.

\section{Preliminary results}\label{section:prel}

From now on, unless otherwise stated, $X$ will denote a pointed metric space with a base point represented by $0$.

For each $x\in X$, the evaluation functional $\delta_x\colon\Lipo(X)\to\R$, defined by $\delta_x (f):=f(x)$ for all $f\in\Lipo(X)$, is linear and continuous with $\left\|\delta_x\right\|=d(x,0)$. 

The Lipschitz-free Banach space over $X$, denoted by $\F(X)$, is the closed subspace of $\Lipo(X)^*$ generated by those evaluation functionals, that is,
$$
\F(X):=\overline{\mathrm{span}}\left\{\delta_x\colon x\in X\right\}\subseteq\Lipo(X)^*.
$$
In the theory of Lipschitz functions, $\F(X)$ is also known under the name of Arens--Ells space and is denoted $\AE(X)$ (see \cite[Chapter 3]{Wea-18}). 

According to \cite[Theorem 3.3]{Wea-18}, $\F(X)$ is a predual Banach space of $\Lipo(X)$. Namely, the evaluation map $Q_X\colon\Lipo(X)\to\F(X)^*$ defined by
$$
Q_X(f)(\gamma):=\gamma(f) \qquad \left(f\in\Lipo(X), \; \gamma\in\F(X)\right),
$$
is the natural isometric isomorphism. 

A normalized elementary molecule is an element of $\F(X)$ in the form
$$
m_{x,y}:=\frac{\delta_x-\delta_y}{d(x,y)}\qquad ((x,y)\in\widetilde{X}).
$$
Note that $||m_{x,y}||=1$ by \cite[Lemma 3.5]{Wea-18}, and it is known (see, for example, \cite[Lemma 2.1]{AliPer-20}) that $\F(X)$ consists of all functionals $\gamma\in\Lipo(X)^*$ of the form
$$
\gamma=\sum_{n=1}^\infty\lambda_n m_{x_n,y_n},
$$
with $\left(\lambda_n\right)_{n=1}^\infty\in\ell_1$ and $\left((x_n,y_n)\right)_{n=1}^\infty\subseteq \widetilde{X}$. Moreover,
$$
\left\|\gamma\right\|=\inf\left\{\sum_{n=1}^\infty\left|\lambda_n\right|\colon \gamma=\sum_{n=1}^\infty\lambda_n m_{x_n,y_n},\, \left(\lambda_n\right)_{n=1}^\infty\in\ell_1,\, \left((x_n,y_n)\right)_{n=1}^\infty\subseteq \widetilde{X}\right\}.
$$

By \cite[Theorem 3.6]{Wea-18}, the map $\delta\colon x\mapsto\delta_x$ is an isometric embedding from $X$ into $\Lipo(X)^*$, and the space $\F(X)$ is characterized up to isometric isomorphism by the following universal extension property: for each Banach space $E$ and each map $f\in\Lipo(X,E)$, there exists a unique operator $T_f\in\L(\F(X),E)$ such that $T_f\circ\delta=f$. Furthermore, $||T_f||=\Lip(f)$.


It is said that $\lipo(X)$ separates points uniformly if there exists a constant $a>1$ such that for every $x,y\in X$, some $f\in\lipo(X)$ satisfies $\Lip(f)\leq a$ and $\left|f(x)-f(y)\right|=d(x,y)$.

Assuming that $X$ is a compact pointed metric space and $\lipo(X)$ separates points uniformly, the restriction map $R_X\colon\F(X)\to\lipo(X)^{*}$ defined by
$$
R_X(\gamma)(f):=\gamma(f) \qquad \left(f\in\lipo(X), \; \gamma\in\F(X)\right), 
$$
is an isometric isomorphism by \cite[Theorem 4.38]{Wea-18}. Examples of compact metric spaces $X$ for which $\lipo(X)$ separates the points uniformly are compact spaces of the form $(X,d\circ \omega)$, where $\omega$ is a local distorsion \cite[Proposition 4.14]{Wea-18} or countable compact metric spaces \cite[Theorem 2.1]{dalet15}.

Next, we will establish quantitative versions of a couple of classical results from Functional Analysis, which connect the injectivity and the surjectivity of an operator in terms of its adjoint operator. In the former result, we will use the ideas of \cite[Hint of Exercise 2.39 (i)]{checos}.

\begin{proposition}\label{prop:caralipsconstant}
Let $E,F$ be Banach spaces and let $T\in\L(E,F)$ be such that $T^*$ is onto (equivalently, $T$ is an into isomorphism). Define
\begin{align*}
M&:=\inf\left\{K>0 \,|\, \forall x^*\in B_{E^*}\; \exists y^*\in KB_{F^*} \colon T^*(y^*)=x^*\right\},\\
N&:=\max\left\{J>0 \colon J\Vert x\Vert\leq \Vert T(x)\Vert,\; \forall x\in E\right\}.
\end{align*}
Then $N=\dfrac{1}{M}$.
\end{proposition}

\begin{proof} 
Set $x\in E\setminus\{0\}$ and take $x^*\in S_{E^*}$ such that $x^*(x)=\Vert x\Vert$. Take $\varepsilon>0$ and, by the definition of $M$, a functional $y^*\in F^*$ such that $T^*(y^*)=x^*$ and $\Vert y^*\Vert\leq M+\varepsilon$. Now we get
$$
\Vert x\Vert=x^*(x)=T^*(y^*)(x)=y^*(T(x))\leq \Vert y^*\Vert \Vert T(x)\Vert\leq (M+\varepsilon)\Vert T(x)\Vert,
$$
and thus $(1/M)\Vert x\Vert\leq \Vert T(x)\Vert$ holds by the arbitrariness of $\varepsilon>0$. The arbitrariness of $x\in E$ implies that 
$$
\frac{1}{M}\in\left\{J>0\colon J\Vert x\Vert\leq \Vert T(x)\Vert,\; \forall x\in E\right\},
$$
and therefore $1/M\leq N$. 

For the reverse inequality, observe that $N\Vert x\Vert\leq \Vert T(x)\Vert$ for all $x\in E$. Select $x^*\in B_{E^*}$ and let us find $y^*\in F^*$ such that $T^*(y^*)=x^*$ and $\Vert y^*\Vert\leq 1/N$. To do so, take $T(E)\subseteq F$, which is a closed subspace since $T$ is an into isomorphism. Consider $T^{-1}\colon T(E)\to E$, which satisfies $\Vert T^{-1}\Vert\leq 1/N$. Consider now the bounded linear functional $x^*\circ T^{-1}\colon T(E)\to\mathbb R$ and note that $\Vert x^*\circ T^{-1}\Vert\leq 1/N$. By Hahn--Banach Theorem there exists $y^*\in F^*$ with $\Vert y^*\Vert \leq 1/N$ and $y^*=x^*\circ T^{-1}$ on $T(E)$. Note that $T^*(y^*)=x^*$ since 
$$
T^*(y^*)(x)=y^*(T(x))=(x^*\circ T^{-1})(T(x))=x^*(x)
$$ 
for all $x\in E$. Consequently, we have proved that 
$$
\frac{1}{N}\in\left\{K>0 \,|\, \forall x^*\in B_{E^*},\; \exists y^*\in KB_{F^*} \colon T^*(y^*)=x^*\right\}.
$$
Hence, by the definition of $M$, we deduce that $M\leq 1/N$ or, equivalently, $N\leq 1/M$.
\end{proof}

\begin{remark}\label{rem:constaopeattaindual}
In the above proof we have proved indeed that the infimum
$$M:=\inf\left\{K>0 \,|\, \forall x^*\in B_{E^*}\; \exists y^*\in KB_{F^*} \colon T^*(y^*)=x^*\right\}$$
is attained since $T^*$ is an adjoint operator, a fact which is well known in Banach space theory.
\end{remark}

The latter result can be seen as a dual version of Proposition~\ref{prop:caralipsconstant}. This time we will follow some ideas from \cite[Hint of Exercise 2.39 (ii)]{checos}.

\begin{proposition}\label{prop:caralipsconstantdual}
Let $E,F$ be Banach spaces and let $T\in\L(E,F)$ be an onto operator (equivalently, $T^*$ is an into isomorphism). Define
\begin{align*}
M&:=\inf\left\{K>0 \,|\, \forall y\in B_{F},\; \exists x\in KB_{E} \colon T(x)=y\right\},\\
N&:=\max\left\{J>0 \colon J\Vert y^*\Vert\leq \Vert T^*(y^*)\Vert,\; \forall y^*\in F^*\right\}.
\end{align*}
Then $N=\dfrac{1}{M}$.
\end{proposition}

\begin{proof}
Let us prove first that $N\geq 1/M$, whose proof will follow the lines of Proposition~\ref{prop:caralipsconstant}.

Let $y^*\in F^*\setminus\{0\}$ and $\varepsilon>0$. We can take $y\in S_F$ such that $y^*(y)>(1-\varepsilon)\Vert y^*\Vert$. By the property defining $M$, we can find $x\in E$ such that $T(x)=y$ and $\Vert x\Vert\leq M+\varepsilon$. Thus
$$
(1-\varepsilon)\Vert y^*\Vert<y^*(y)=y^*(T(x))=T^*(y^*)(x)\leq \Vert T^*(y^*)\Vert \Vert x\Vert\leq (M+\varepsilon)\Vert T^*(y^*)\Vert.
$$
So $((1-\varepsilon)/(M+\varepsilon))\Vert y^*\Vert\leq \Vert T^*(y^*)\Vert$ for every $y^*\in F^*$. Then the definition of $N$ yields $((1-\varepsilon)/(M+\varepsilon))\leq N$ for all $\varepsilon>0$, and so $M\geq 1/N$. 

In order to prove the equality, assume by contradiction that $M>1/N$, and find $\varepsilon>0$ such that $M>1/(N-\varepsilon)$. Consequently, we can take a point $y\in B_F$ satisfying the condition:
$$
T(x)=y \text{ for some }x\in E \quad \Rightarrow \quad \Vert x\Vert>\frac{1}{N-\varepsilon}.
$$
Given $\eta>0$ small enough so that 
$$
\frac{1}{N-\frac{\varepsilon}{2}}<\frac{1}{N-\varepsilon}-(M+1)\eta,
$$
we claim that 
$$
\left(y+\eta B_F\right)\cap \left(\frac{1}{N-\varepsilon}-(M+1)\eta\right)T(B_E)=\emptyset.
$$
Indeed, given $y'\in B_F$, take $x\in E$ such that $T(x)=y+\eta y'$. By the property defining $M$, there exists $x'\in E$ such that $T(x')=y'$ and $\Vert x'\Vert\leq M+1$. Now $y=T(x-\eta x')$ and so $\Vert x-\eta x'\Vert>1/(N-\varepsilon)$. Hence
$$
\Vert x\Vert>\frac{1}{N-\varepsilon}-\eta\Vert x'\Vert\geq \frac{1}{N-\varepsilon}-(M+1)\eta.
$$
This shows that $y+\eta y'\notin(1/(N-\varepsilon)-(M+1)\eta)T(B_E)$, and the arbitrariness of $y'\in B_F$ proves our claim. 

We now prove that $y\notin (1/(N-\varepsilon/2))\overline{T(B_E)}$. Otherwise, we could find a sequence $(y_n)_{n=1}^\infty$ in $(1/(N-\varepsilon/2))B_E$ such that $(T(y_n))_{n=1}^\infty$ converges to $y$. In light of our claim, note that $T(y_n)\neq y$ for all $n\in\N$ since $y_n\in(1/(N-\varepsilon)-(M+1)\eta)B_E$ for all $n\in\N$. Hence, for each $n\in\N$, we can take a point $x'_n\in E$ with $\left\|x_n'\right\|\leq (M+1)$ such that $T(\left\|y-T(y_n)\right\|x'_n)=y-T(y_n)$, and thus $T(y_n+||y-T(y_n)||x'_n)=y$. Write $v_n=y_n+||y-T(y_n)||x'_n$ and note that  
$$
\left\|v_n\right\|\leq \frac{1}{N-\frac{\varepsilon}{2}}+\left\|y-T(y_n)\right\|(M+1)\rightarrow\frac{1}{N-\frac{\varepsilon}{2}}\qquad (n\to+\infty).
$$
Since $1/(N-\varepsilon/2)<1/(N-\varepsilon)-(M+1)\eta$, we can find some $n\in\N$ for which $||v_n||<1/(N-\varepsilon)-(M+1)\eta$ and $T(v_n)=y$, and this contradicts the condition that $y$ must satisfy. 

Finally, by Hahn--Banach Theorem we can find $y^*\in F^*$ and $\beta\in\R$ such that
$$
\sup_{x\in B_E}\frac{1}{N-\frac{\varepsilon}{2}} y^*(T(x))<\beta\leq y^*(y)\leq \Vert y^*\Vert \Vert y\Vert\leq\Vert y^*\Vert,
$$
and thus 
$$
\Vert T^*(y^*)\Vert=\sup_{x\in B_E}T^*(y^*)(x)=\sup_{x\in B_E} y^*(T(x))<
\left(N-\frac{\varepsilon}{2}\right)\Vert y^*\Vert,
$$
which entails in a contradiction with the property defining $N$. This proves that $M=1/N$ as desired.
\end{proof}

\begin{remark}\label{non-proximinalremark}
In Proposition \ref{prop:caralipsconstantdual}, we can not guarantee that the infimum $M$ is indeed attained. 
\end{remark}

Our study will also require a brief appeal to the theory of tensor products. Let us recall that the projective tensor product of Banach spaces $E$ and $F$, denoted by $E \pten F$, is the completion of the space $E\otimes F$, endowed with the projective norm  
\begin{align*}
\pi(z) &= \inf \left\{ \sum_{n=1}^{\infty} \|x_n\| \|y_n\|: \sum_{n=1}^{\infty} \|x_n\| \|y_n\| < \infty, z = \sum_{n=1}^{\infty} x_n \otimes y_n \right\} \\
&= \inf \left\{ \sum_{n=1}^{\infty} |\lambda_n|: z = \sum_{n=1}^{\infty} \lambda_n x_n \otimes y_n, \sum_{n=1}^{\infty} |\lambda_n| < \infty, \|x_n\| = \|y_n\| = 1 \right\},
\end{align*}
where the infinum is taken over all such representations of $z$. It is well-known that $\pi(x\otimes y)= \|x\| \|y\|$ for every $x \otimes y\in E\pten F$, and the closed unit ball of $E \pten F$ is the closed convex hull of the set $B_E \otimes B_F := \{ x \otimes y\colon x \in B_E,\, y \in B_F \}$. We will also apply the identification $(E \pten F)^*\cong\L(E,F^*)$, where the action of an operator $T\in\L(E,F^*)$ as a linear functional on $E\pten F$ is given by
$$
T\left(\sum_{n=1}^{\infty} x_n \otimes y_n \right) = \sum_{n=1}^{\infty} T(x_n)(y_n),
$$
for every $\sum_{n=1}^{\infty} x_n \otimes y_n \in E\pten F$. The reader is referred to Chapter 2 of the excellent book \cite{ryan} for more background on projective tensor products.


\section{Lipschitz interpolating sets}\label{section:Lipschitzseq}

Let us begin with some basic examples of Lipschitz interpolating finite sets.

\begin{example}\label{ex-0}
Let $X$ be a pointed metric space with at least two points. If $(x,y)\in\widetilde{X}$, for each $\alpha\in\R$, the function $f\colon X\to\R$, given by 
$$
f(z)=\alpha(d(z,y)-d(0,y))\qquad (z\in X),
$$
belongs to $\Lipo(X)$ with $\left\|f\right\|=\left|\alpha\right|$ and $\displaystyle\frac{f(x)-f(y)}{d(x,y)}=\alpha$. 
\end{example}

In the case of a set $\left((x_i,y_i)\right)_{i\in I}$ with $\#(I)\geq 2$, the following example shows that the conditions:
\begin{enumerate}
	\item $\{x_i,y_i\}\cap \{x_j,y_j\}=\emptyset$ whenever $i,j\in I$ and $i\neq j$,
	\item $x_i\neq 0\neq y_i$ for all $i\in I$,
\end{enumerate}
are relevant for the set $\left((x_i,y_i)\right)_{i\in I}$ in $\widetilde{X}$ to be Lipschitz interpolating for $\Lipo(X)$.

\begin{example}\label{ex-2}
Take any pointed metric space $X$ with at least three points and select $x,y,0\in X$ different from each other. Then the set $\left\{(x,0),(y,0),(x,y)\right\}$ is not Lipschitz interpolating for $\Lipo(X)$ since the operator $T:\Lipo(X)\to \ell_\infty^3$ defined by
$$
T(f)=\left(\frac{f(x)}{d(x,0)}, \frac{f(y)}{d(y,0)},\frac{f(x)-f(y)}{d(x,y)} \right),
$$
is not onto. Observe that there is no $f$ such that $T(f)=(1,0,0)$. 
\end{example}

In light of Example \ref{ex-2}, we proceed to generalize Example \ref{ex-0} as follows.

\begin{example}\label{ex-3}
Let $X$ be a pointed metric space and assume that $\left((x_i,y_i)\right)_{i=1}^n$ for some $n\in\N$ is a finite sequence in $\widetilde{X}$ such that $\{x_i,y_i\}\cap \{x_j,y_j\}=\emptyset$ if $i,j\in\{1,\ldots,n\}$ and $i\neq j$, and $x_i\neq 0\neq y_i$ for all $i\in\{1,\ldots,n\}$. Consider the set 
$$
X_n=\{x_i\colon i=1,\ldots,n\}\cup\{0\}\cup \{y_i\colon i=1,\ldots,n\}.
$$
Given $\left(\alpha_i\right)_{i=1}^n$ in $\R$, for $i=1,\ldots,n$ define the function $g_i\colon X_n\to\R$ by 
$$
g_i(z)=\left\{
\begin{array}{lll}
         \alpha_id(z,y_i)& \text{ if } & z=x_i,\\
         &           &   \\
	0& \ & \text{ otherwise.}
	\end{array}
\right.
$$
Clearly, $g_i\in\Lipo(X_n)$ with 
$$
\left\|g_i\right\|=\frac{\left|\alpha_i\right|d(x_i,y_i)}{d(x_i,X_n\setminus\{x_i\})}.
$$
An application of McShane extension Theorem (see, for example, \cite[Theorem 1.33]{Wea-18}) yields a function $f_i\in\Lipo(X)$ such that $f_i|_{X_n}=g_i$ and $\left\|f_i\right\|=\left\|g_i\right\|$. Clearly, the function $f=\sum_{i=1}^nf_i$ is in $\Lipo(X)$ with 
$$
\left\|f\right\|\leq \sum_{i=1}^n\frac{\left|\alpha_i\right|d(x_i,y_i)}{d(x_i,X_n\setminus\{x_i\})},
$$
and since the supports of the functions $g_i$ are pairwise disjoint, we conclude that 
$$
\left(\displaystyle\frac{f(x_i)-f(y_i)}{d(x_i,y_i)}\right)_{i=1}^n=\left(\alpha_i\right)_{i=1}^n.
$$
\end{example}

For infinite metric spaces, more examples of non Lipschitz interpolating sequences can be given:

\begin{example}
Consider a pointed metric space $X$ containing a sequence $(x_n)_{n=1}^\infty$ of pairwise distinct points in $X\setminus\{0\}$ that converges to some $x\in X\setminus\{0\}$. Then the sequence $\left((x_n,0)\right)_{n=1}^\infty$ is not Lipschitz interpolating for $\Lipo(X)$. Indeed, otherwise, we could have a function $f\in\Lipo(X)$ such that  
$$
\left(\frac{f(x_n)}{d(x_n,0)}\right)_{n=1}^\infty=\left((-1)^n\right)_{n=1}^\infty .
$$
Hence $f(x_{2n})=d(x_{2n},0)$ and $f(x_{2n+1})=-d(x_{2n+1},0)$ for all $n\in\N$, taking limits with $n\to\infty$ yields $f(x)=d(x,0)=-d(x,0)$, and thus $x=0$, a contradiction.
\end{example}

In order to find an explanation for the above behaviour, let us take a more global vision on Lipschitz interpolating operators.

Given a set $\left((x_i,y_i)\right)_{i\in I}$ in $\widetilde{X}$, its associate Lipschitz interpolating operator $T:\Lipo(X)\to\ell_\infty(I)$ is an adjoint operator. Indeed, an easy verification yields the following result.

\begin{proposition}\label{prop-new}
Let $\left((x_i,y_i)\right)_{i\in I}$ be a set in $\widetilde{X}$. Then the operator $S:\ell_1(I)\to\mathcal F(X)$ defined by
$$
S(\lambda)=\sum_{i\in I}\lambda_im_{x_i,y_i}\qquad (\lambda=(\lambda_i)_{i\in I}\in\ell_1(I)),
$$
is linear and continuous with $\left\|S\right\|\leq 1$, and $S^*=T$ (up to the identifications $\mathcal F(X)^*\cong\Lipo(X)$ and $\ell_1(I)^*\cong\ell_\infty(I)$). $\hfill \square$
\end{proposition}

If in addition the set $\left((x_i,y_i)\right)_{i\in I}$ is Lipschitz interpolating for $\Lip_0(X)$ with Lipschitz interpolation constant $M$, then Proposition~\ref{prop:caralipsconstant} shows that $(1/M)\left\|\lambda\right\|_1\leq\left\|S(\lambda)\right\|$ for all $\lambda\in \ell_1(I)$. Moreover, $1/M$ is the biggest positive constant for which the preceding inequality holds. Taking also into account Remark~\ref{rem:constaopeattaindual}, the following result is proved.

\begin{theorem}\label{theo:caraLipschitintergeneral}
Let $\left((x_i,y_i)\right)_{i\in I}$ be a set in $\widetilde{X}$ and let $M\geq 1$. The following statements are equivalent:
\begin{enumerate}
\item The set $\left((x_i,y_i)\right)_{i\in I}$ is Lipschitz interpolating for $\Lipo(X)$ and its Lipschitz interpolating constant is $M$.
\item The operator $S:\ell_1(I)\to\mathcal F(X)$ defined by $S(e_i)=m_{x_i,y_i}$ is an into isomorphism and
$$
\frac{1}{M}=\max\{J>0\colon J\Vert \lambda\Vert_1\leq \Vert S(\lambda)\Vert,\; \forall \lambda\in\ell_1(I)\}.
$$
As a consequence, $M:=\inf\{K\geq 1: \forall \alpha\in B_{\ell_\infty(I)}, \exists\ f\in KB_{\Lipo(X)}: T(f)=\alpha\}$. $\hfill \square$
\end{enumerate}
\end{theorem}

Let us now consider the following example -- based on \cite[Lemma 3.1]{amrt24} -- that can serve as a starting point for building our theory. As usual, we denote by $B(x,r)$ the open ball in the metric space $(X,d)$ with center $x\in X$ and radius $r>0$. 

\begin{example}\label{ex-1}
Let $X$ be a pointed metric space and assume that $\left((x_i,y_i)\right)_{i\in I}$ is a subset of $\widetilde{X}$ satisfying the following conditions:
\begin{enumerate}
    \item There are positive numbers $r_i>0$ such that $B(x_i,r_i)\cap B(x_j,r_j)=\emptyset$ if $i,j\in I$ and $i\neq j$.
    \item For every $i,j\in I$ with $i\neq j$, we have
    $$\frac{r_i+r_j}{d(B(x_i,r_i), B(x_j,r_j))}<\frac{1}{2}.$$
    \item For every $i\in I$, it holds that $0<d(x_i,y_i)<d(y_i,X\setminus B(x_i,r_i))$.
\end{enumerate}
Then $\left((x_i,y_i)\right)_{i\in I}$ is a Lipschitz interpolating set for $\Lipo(X)$ and its Lipschitz interpolation constant is $1$.
\end{example}

\begin{proof}
We first construct, for each $i\in I$, a function $f_i\in S_{\Lipo(X)}$ satisfying that $m_{x_j,y_j}(f_i)=\delta_{ij}$ for any $j\in I$. Indeed, consider $f_i\colon (X\setminus B(x_i,r_i))\cup\{x_i,y_i\}\to\mathbb{R}$ given by 
$$f_i(z)=\left\{\begin{array}{ll}
   0  & \mbox{if }z\neq y_i, \\
   -d(x_i,z)  & \mbox{if }z=y_i.
\end{array} \right.$$
Observe that $\Vert f_i\Vert=1$ since we have 
$$
\frac{f_i(z)-f_i(y_i)}{d(z,y_i)}=\frac{d(x_i,y_i)}{d(z,y_i)}\leq \frac{d(x_i,y_i)}{d(y_i,X\setminus B(x_i,r_i))}<1
$$
 and  
$$
\frac{f_i(x_i)-f_i(y_i)}{d(x_i,y_i)}=\frac{d(x_i,y_i)}{d(x_i,y_i)}=1.
$$
Hence we can extend $f_i$ by McShane Theorem to a norm-one function $f_i\in\Lipo(X)$. 

Now, given $\alpha=(\alpha_i)_{i\in I}\in \ell_\infty(I)$, consider the function $f:X\to\mathbb R$ given by
$$
f(z)=\sum_{i\in I}\alpha_i f_i(z),
$$
which is well defined 
since the supports of the functions $f_i$ are pairwise disjoint.
Let us prove that $f\in\Lipo(X)$ and $\Vert f\Vert=\Vert \alpha\Vert_\infty$. To do so, let $x,y\in X$ with $x\neq y$. Observe that if $x,y\notin\bigcup\limits_{i\in I}B(x_i,r_i)$, then $f(x)=f(y)=0$, hence $f(x)-f(y)=0$, so we assume $x\in B(x_i,r_i)$, which implies $f(x)=\alpha_i f_i(x)$. Now we have different possibilities for the position of $y$:
\begin{enumerate}
    \item[(1)] If $y\notin \bigcup\limits_{j\in I}B(x_j,r_j)$, then $f(y)=0$ and therefore
    \[\begin{split}
    f(x)-f(y)=\alpha_i f_i(x)=\alpha_i (f_i(x)-f_i(y))\leq \vert \alpha_i\vert \Vert f_i\Vert d(x,y)\leq \Vert \alpha\Vert_{\infty}d(x,y).
    \end{split}\]
    \item[(2)] If $y\in B(x_i,r_i)$ then $f(y)=\alpha_i f_i(y)$ and the conclusion that $f(x)-f(y)\leq \left\|\alpha\right\|_\infty d(x,y)$ is similar in this case.
    \item[(3)] If $y\in \bigcup\limits_{j\neq i}B(x_j,r_j)$, then $f(y)=\alpha_j f_j(y)$ for some $j\neq i$. In this case, using that $f_i(x_i)=f_j(x_j)=0$ by definition, we get 
    \[\begin{split}
    f(x)-f(y)=\alpha_i f_i(x)-\alpha_j f_j(y)
		& =\alpha_i (f_i(x)-f_i(x_i))+\alpha_j(f_j(x_j)-f_j(y))\\
    & \leq \vert \alpha_i\vert \Vert f_i\Vert d(x,x_i)+\vert \alpha_j\vert \Vert f_j\Vert d(x_j,y)\\
    & = \vert \alpha_i\vert  d(x,x_i)+\vert\alpha_j\vert   d(x_j,y)\\
    & \leq \left\|\alpha\right\|_\infty (d(x_i,x)+d(x_j,y))\\
    & \leq \left\|\alpha\right\|_\infty (r_i+r_j)\\
    & \leq \frac{\left\|\alpha\right\|_\infty}{2} d(B(x_i,r_i),B(x_j,r_j))\\
    & \leq \frac{\left\|\alpha\right\|_\infty}{2}d(x,y).
    \end{split}\]
\end{enumerate}
All the above discussion proves that $f\in \Lipo(X)$ and $\Vert f\Vert\leq \left\|\alpha\right\|_\infty$. To prove the reverse inequality, simply note that, given $i\in I$, it follows that
$$
\vert \alpha_i\vert=\frac{f(x_i)-f(y_i)}{d(x_i,y_i)}\leq \Vert f\Vert,
$$
and taking supremum over all $i\in I$ yields $\left\|\alpha\right\|_\infty\leq \Vert f\Vert$.

In order to prove that $\left((x_i,y_i)\right)_{i\in I}$ is a Lipschitz interpolating set for $\Lipo(X)$ whose Lipschitz interpolating constant is $1$, we will prove in virtue of Proposition~\ref{prop:caralipsconstant} that the operator $S:\ell_1(I)\to \mathcal F(X)$, defined by $S(\lambda)=\sum_{i\in I}\lambda_i m_{x_i,y_i}$ for all $\lambda=(\lambda_i)_{i\in I}\in\ell_1(I)$, is an into isometry. In order to do so, take any $\lambda=(\lambda_i)_{i\in I}\in\ell_1(I)$ of finite support (which is dense in $\ell_1(I)$) and select, for every $i\in I$, $\alpha_i:=\sign(\lambda_i)$, with the usual convention $\sign(0)=0$. Then $(\alpha_i)_{i\in I}$ is a set of finite support of $c_0(I)$. By the above, the function $f=\sum_{i\in I}\alpha_i f_i$ is in $S_{\Lipo(X)}$. Furthermore, 
\begin{align*}
\Vert S(\lambda)\Vert
&\geq f\left(\sum_{i\in I}\lambda_i m_{x_i,y_i}\right)
=\sum_{i\in I}\lambda_i f(m_{x_i,y_i})
=\sum_{i\in I}\sum_{j\in I}\alpha_j \lambda_i f_j(m_{x_i,y_i})\\
&=\sum_{i\in I}\lambda_i\alpha_i
=\sum_{i\in I}\vert \lambda_i\vert
=\Vert \lambda\Vert_1,
\end{align*}
as desired.
\end{proof}

A look at Example \ref{ex-1} shows that the following fact was essential: if $T:\Lipo(X)\to\ell_\infty(I)$ is the Lipschitz interpolating operator associated to $\left((x_i,y_i)\right)_{i\in I}$, then there exists a set of functions $(f_i)_{i\in I}$ in $S_{\Lipo(X)}$ such that $\overline{\mathrm{span}}((f_i)_{i\in I})$ is isometric to $\ell_\infty(I)$ and $T(f_i)=e_i$ holds for every $i\in I$. 

This motivates the introduction of the following sets of functions in $\Lipo(X)$ with a terminology borrowed from an elegant result of uniform algebra theory -- due to Beurling -- that appears in the Garnett's monograph \cite[Chapter VII, Theorem 2.1]{Gar-81}.

\begin{definition}\label{def:Berurlingeq}
Let $X$ be a pointed metric space and let $\left((x_i,y_i)\right)_{i\in I}$ be a Lipschitz interpolating set in $\widetilde{X}$ for $\Lipo(X)$ with Lipschitz interpolation constant $M$. A set $(f_i)_{i\in I}$ of real-valued-functions defined on $X$ is a Beurling set of functions in $\Lipo(X)$ for $\left((x_i,y_i)\right)_{i\in I}$ if $f_i(0)=0$ for all $i\in I$, $m_{x_j,y_j}(f_i)=\delta_{ij}$ for any $i,j\in I$, and 
$$
\sup_{(x,y)\in\widetilde{X}}\sum_{i\in I}\frac{\vert f_i(x)-f_i(y)\vert}{d(x,y)}\leq M.
$$ 
\end{definition}

Continuing with Example \ref{ex-1}, an easy verification shows that $R:\ell_\infty(I)\to\Lipo(X)$ given by  
$$
R(\alpha)=\sum_{i\in I}\alpha_i f_i,
$$
for any $\alpha=(\alpha_i)_{i\in I}\in\ell_\infty(I)$, is an isometric linear lifting of the operator $T:\Lipo(X)\to\ell_\infty(I)$ in the sense that $T\circ R=\mathrm{Id}_{\ell_\infty(I)}$ and $\Vert R\Vert=1$, which is precisely the Lipschitz interpolation constant of the set $\left((x_i,y_i)\right)_{i\in I}$. This fact is a reflection of the following more general result.

\begin{theorem}\label{rem:levanvarapouloset}
Let $\left((x_i,y_i)\right)_{i\in I}$ be a Lipschitz interpolating set in $\widetilde{X}$ for $\Lipo(X)$ with Lipschitz interpolating constant $M$. Assume that $(f_i)_{i\in I}$ is a Beurling set of functions in $\Lipo(X)$ for $\left((x_i,y_i)\right)_{i\in I}$. Then the mapping $R:\ell_\infty(I)\to\Lipo(X)$ given by
$$
R(\alpha)=\sum_{i\in I}\alpha_i f_i \qquad (\alpha=(\alpha_i)_{i\in I}\in\ell_\infty(I)),
$$
is linear and continuous with $\left\|R\right\|=M$ and $T\circ R=\mathrm{Id}_{\ell_\infty(I)}$.
\end{theorem}

\begin{proof}
Observe first that $R$ is well defined in the sense that, given $\alpha=(\alpha_i)_{i\in I}\in\ell_\infty(I)$, the formal sum $\sum_{i\in I} \alpha_i f_i$ is pointwise well defined. Indeed, given $x\in X\setminus\{0\}$ and $F\subseteq I$ finite, we have
\begin{align*}
\sum_{i\in F}\vert \alpha_i f_i(x)\vert
&=\sum_{i\in F}\vert \alpha_i(f_i(x)-f_i(0))\vert\\
&\leq d(x,0)\Vert \alpha\Vert_\infty \sum_{i\in F}\frac{\vert f_i(x)-f_i(0)\vert}{d(x,0)}\\
&\leq d(x,0)\Vert \alpha\Vert_\infty M.
\end{align*}
This implies that the family $\sum_{i\in I}\alpha_i f_i(x)$ is absolutely summable, so it is summable by a completeness argument and thus the function $R(\alpha)\colon X\to\mathbb R$ is pointwise well defined. 

Let $\alpha=(\alpha_i)_{i\in I}\in B_{\ell_\infty(I)}$ and define $g=\sum_{i\in I}\alpha_i f_i$. Clearly, $g(0)=0$. For any $(x,y)\in\widetilde{X}$, we have
\begin{align*}
\frac{\vert g(x)-g(y)\vert}{d(x,y)}
&=\left|\sum_{i\in I}\frac{\alpha_i(f_i(x)-f_i(y))}{d(x,y)}\right|
\leq \sum_{i\in I}\frac{\left|\alpha_i\right|\left|f_i(x)-f_i(y)\right|}{d(x,y)}\\
&\leq \Vert \alpha\Vert_\infty \sup_{(u,v)\in\widetilde{X}}\sum_{i\in I}\frac{\vert f_i(u)-f_i(v)\vert}{d(u,v)}
\leq M,   
\end{align*}
and therefore $g\in\Lipo(X)$ with $\Vert g\Vert\leq M$. Since $R\colon\ell_\infty(I)\to\Lipo(X)$ is clearly linear, it follows that $R\colon\ell_\infty(I)\to\Lipo(X)$ is continuous with $\Vert R\Vert\leq M$. Now, the condition that $m_{x_j,y_j}(f_i)=\delta_{ij}$ for any $i,j\in I$ implies that $T\circ R=\mathrm{Id}_{\ell_\infty(I)}$. Indeed, given $\alpha=(\alpha_i)_{i\in I}\in\ell_\infty(I)$, one has 
\begin{align*}
(T\circ R)(\alpha)
&=T\left(\sum_{i\in I}\alpha_i f_i\right)=\left(\sum_{i\in I}\frac{\alpha_i (f_i(x_j)-f_i(y_j))}{d(x_j,y_j)}\right)_{j\in I}\\
&=\left(\sum_{i\in I}\alpha_i m_{x_j,y_j}(f_i)\right)_{j\in I}=\left(\sum_{i\in I}\alpha_i\delta_{ij}\right)_{j\in I}=\left(\alpha_j\right)_{j\in I}=\alpha.
\end{align*}
Finally, note that $\Vert R\Vert=M$ since given any $\alpha\in B_{\ell_\infty(I)}$, we have that $T(R(\alpha))=\alpha$ where $R(\alpha)\in\Lipo(X)$ with $\left\|R(\alpha)\right\|\leq\left\|R\right\|$, and therefore $M\leq\Vert R\Vert$ by Definition \ref{def-LIC} and Remark \ref{rem:constaopeattaindual}.
\end{proof}

A kind of converse of Theorem \ref{rem:levanvarapouloset} can be established as follows.

\begin{theorem}\label{cor:caravarapoulos0}
Let $\left((x_i,y_i)\right)_{i\in I}$ be a Lipschitz interpolating set in $\widetilde{X}$ for $\Lipo(X)$ and let $M$ be its Lipschitz interpolating constant. Assume that there exists $R\in\L(\ell_\infty(I),\Lipo(X))$ such that $\left\|R\right\|\leq M$ and $T\circ R=\mathrm{Id}_{\ell_\infty(I)}$. Then there exists a Beurling set $(f_i)_{i\in I}$ of functions in $\Lipo(X)$ for $\left((x_i,y_i)\right)_{i\in I}$.
\end{theorem}

\begin{proof}
For each $i\in I$, define $f_i=R(e_i)\in\Lipo(X)$. Firstly, for any $i,j\in I$, we have
$$
m_{x_j,y_j}(f_i)=\frac{f_i(x_j)-f_i(y_j)}{d(x_j,y_j)}=\frac{R(e_i)(x_j)-R(e_i)(y_j)}{d(x_j,y_j)}=T(R(e_i))(j)=e_i(j)=\delta_{ij}.
$$
Secondly, let $(x,y)\in \widetilde{X}$. For every $i\in I$, there exists $\alpha_i\in\{-1,1\}$ such that $\left|f_i(x)-f_i(y)\right|=\alpha_i(f_i(x)-f_i(y))$. Given a finite set $F\subseteq I$, take $\beta_F=(\beta_i)_{i\in I}\in\ell_\infty(I)$ defined by 
$$
\beta_i=\left\{
\begin{array}{lll}
\alpha_i& \text{ if } & i\in F,\\
        &             &        \\
0       & \text{ if } & i\in I\setminus F.
\end{array}
\right.
$$
Note that $R(\beta_F)=\sum_{i\in F}\alpha_if_i$ and 
$$
\sum_{i\in F}\frac{\left|f_i(x)-f_i(y)\right|}{d(x,y)}
=\sum_{i\in F}\frac{\alpha_if_i(x)-\alpha_if_i(y)}{d(x,y)}
=m_{x,y}(R(\beta_F))\leq\left\|R(\beta_F)\right\|\leq \left\|R\right\|\left\|\beta_F\right\|_\infty\leq M.
$$
Therefore $\displaystyle\sum_{i\in I}\frac{\left|f_i(x)-f_i(y)\right|}{d(x,y)}\leq M$ by the arbitrariness of the finite set $F\subseteq I$. Since $(x,y)$ was arbitrary in $\widetilde{X}$, we deduce that 
$$
\sup_{(x,y)\in\widehat{X}}\sum_{i\in I}\frac{\vert f_i(x)-f_i(y)\vert}{d(u,v)}\leq M,
$$   
and this completes the proof.
\end{proof}

The following theorem shows that the existence of a lifting of the Lipschitz interpolating operator $T$ goes beyond in a preadjoint setting.

\begin{theorem}\label{theo:varapoadjointope} 
Let $\left((x_i,y_i)\right)_{i\in I}$ be a Lipschitz interpolating set in $\widetilde{X}$, let $M$ be its Lipschitz interpolation constant and let $(f_i)_{i\in I}$ be a Beurling set of functions in $\Lipo(X)$ for $\left((x_i,y_i)\right)_{i\in I}$. Then the bounded linear operator $R:\ell_\infty(I)\to \Lipo(X)$, defined by
$$
R(\alpha)=\sum_{i\in I}\alpha_i f_i\qquad (\alpha=(\alpha_i)_{i\in I}\in\ell_\infty(I)),
$$
is weak*-to-weak* continuous.
\end{theorem}

\begin{proof}
Note that $R$ is weak*-to-weak* continuous if, and only if, for every $\gamma\in\mathcal F(X)$ we get that $\gamma\circ R: \ell_\infty(I)\to \mathbb R$ is a $w^*$-continuous functional. By Banach--Dieudonn\'e Theorem \cite[Corollary 4.46]{checos}, this is equivalent to proving that $\ker(\gamma\circ  R)\cap B_{\ell_\infty(I)}$ is $w^*$-closed.

In order to prove it, take a $w^*$-convergent net $(\alpha_s)_s$ in $\ker(\gamma\circ  R)\cap B_{\ell_\infty(I)}$ with $\alpha_s=(\alpha_{s,i})_{i\in I}$ for each $s$, which converges to $\alpha=(\alpha_i)_{i\in I}\in\ell_\infty(I)$. Since the norm is $w^*$-lower semicontinuous, it follows that $\Vert\alpha\Vert_\infty\leq \liminf_{s}\Vert \alpha_s\Vert_\infty\leq 1$. Let us prove that $(\gamma\circ R)(\alpha)=0$. In order to do so, select $\varepsilon>0$ and let us prove that $\vert (\gamma\circ R)(\alpha)\vert<\varepsilon$. 

To begin with observe that we can find $\lambda_n\in\mathbb R^+$ and $(x_n,y_n)\in\widetilde{X}$ for every $n\in\mathbb N$ such that $\gamma=\sum_{n=1}^\infty \lambda_n m_{x_n,y_n}$ and $\sum_{n=1}^\infty \vert\lambda_n\vert<\infty$. 

Now select $n_0\in \mathbb N$ such that if $\gamma_0:=\sum_{n=1}^{n_0}\lambda_n m_{x_n,y_n}$, we get $\Vert \gamma-\gamma_0\Vert<\varepsilon/6M$. Then
$$
\vert (\gamma\circ R)(\alpha)\vert
=\vert  \gamma(R(\alpha))\vert
\leq \vert  \gamma_0(R(\alpha))\vert+\left\|\gamma-\gamma_0\right\|\left\|R\right\|\left\|\alpha\right\|_\infty
\leq \frac{\varepsilon}{6}+\vert\gamma_0(R(\alpha))\vert.
$$
Now observe that
$$
\gamma_0(R(\alpha))=\sum_{n=1}^{n_0}\lambda_n \frac{ R(\alpha)(x_n)- R(\alpha)(y_n)}{d(x_n,y_n)}.
$$
For every $1\leq n\leq n_0$, we get that $\sum_{i\in I} \vert f_i(x_n)\vert<\infty$ and $\sum_{i\in I}\vert f_i(y_n)\vert<\infty$ hold (Theorem~\ref{rem:levanvarapouloset}). Consequently, we can find a finite set $F\subseteq I$ such that, for every $1\leq n\leq n_0$,  
$$
\sum_{i\notin F}\vert f_i(x_n)\vert+\sum_{i\notin F}\vert f_i(y_n)\vert<\frac{\varepsilon  d(x_n,y_n)}{6\lambda_n n_0}.
$$
Thus
\begin{align*}
\vert \gamma_0(R(\alpha))\vert
& =\left\vert \sum_{n=1}^{n_0}\lambda_n \frac{\sum_{i\in I}\alpha_if_i(x_n)-\sum_{i\in I}\alpha_i f_i(y_n)}{d(x_n,y_n)} \right\vert\\
& \leq \left\vert \sum_{n=1}^{n_0}\lambda_n \frac{\sum_{i\in F}\alpha_if_i(x_n)-\sum_{i\in F}\alpha_i f_i(y_n)}{d(x_n,y_n)} \right\vert +\left\vert \sum_{n=1}^{n_0}\lambda_n \frac{\sum_{i\notin F}\alpha_if_i(x_n)-\sum_{i\notin F}\alpha_i f_i(y_n)}{d(x_n,y_n)} \right\vert\\
& \leq \left\vert \sum_{n=1}^{n_0}\lambda_n \frac{\sum_{i\in F}\alpha_if_i(x_n)-\sum_{i\in F}\alpha_i f_i(y_n)}{d(x_n,y_n)} \right\vert+\sum_{n=1}^{n_0} \lambda_n \frac{\sum_{i\notin F}\vert f_i(x_n)\vert+\sum_{i\notin F}\vert f_i(y_n)\vert}{d(x_n,y_n)}\\
& < \left\vert \sum_{n=1}^{n_0}\lambda_n \frac{\sum_{i\in F}\alpha_if_i(x_n)-\sum_{i\in F}\alpha_i f_i(y_n)}{d(x_n,y_n)} \right\vert+\frac{\varepsilon}{6}.
\end{align*}
Hence
$$\vert (\gamma\circ R)(\alpha)\vert<\frac{\varepsilon}{3}+\left\vert \sum_{n=1}^{n_0}\lambda_n \frac{\sum_{i\in F}\alpha_if_i(x_n)-\sum_{i\in F}\alpha_i f_i(y_n)}{d(x_n,y_n)} \right\vert.$$
Now we use that the net $(\alpha_s)_s$ is $w^*$-convergent to $\alpha$, which means $(\alpha_{s,i})_s\rightarrow \alpha_i$ holds for every $i\in I$. Consequently, we can find $s$ large enough to guarantee
$$
\vert \alpha_{s,i}-\alpha_i\vert<\frac{\varepsilon }{6 M n_0\max\{\lambda_1,\ldots, \lambda_{n_0}\}}.
$$
Hence
\begin{align*}
   &\left\vert \sum_{n=1}^{n_0}\lambda_n \frac{\sum_{i\in F}\alpha_if_i(x_n)-\sum_{i\in F}\alpha_i f_i(y_n)}{d(x_n,y_n)} \right\vert\\
	 &\leq \left\vert \sum_{n=1}^{n_0}\lambda_n \frac{\sum_{i\in F}\alpha_{s,i}f_i(x_n)-\sum_{i\in F}\alpha_{s,i} f_i(y_n)}{d(x_n,y_n)} \right\vert 
    +\left\vert \sum_{n=1}^{n_0}\lambda_n \frac{\sum_{i\in F}(\alpha_i-\alpha_{s,i})f_i(x_n)-\sum_{i\in F}(\alpha_i-\alpha_{s,i}) f_i(y_n)}{d(x_n,y_n)} \right\vert\\
   &\leq  \left\vert \sum_{n=1}^{n_0}\lambda_n \frac{\sum_{i\in F}\alpha_{s,i}f_i(x_n)-\sum_{i\in F}\alpha_{s,i} f_i(y_n)}{d(x_n,y_n)} \right\vert
   +\sum_{n=1}^{n_0}\lambda_n  \sum_{i\in F}\vert \alpha_i-\alpha_{s,i}\vert\frac{\vert f_i(x_n)-f_i(y_n)\vert}{d(x_n,y_n)}\\
   & \leq \left\vert \sum_{n=1}^{n_0}\lambda_n \frac{\sum_{i\in F}\alpha_{s,i}f_i(x_n)-\sum_{i\in F}\alpha_{s,i} f_i(y_n)}{d(x_n,y_n)} \right\vert
   + \sum_{n=1}^{n_0}\lambda_n \sum_{i\in F}\frac{\varepsilon }{6 M n_0\max\{\lambda_1,\ldots, \lambda_{n_0}\}}\frac{\vert f_i(x_n)-f_i(y_n)\vert}{d(x_n,y_n)}\\
   & \leq \left\vert \sum_{n=1}^{n_0}\lambda_n \frac{\sum_{i\in F}\alpha_{s,i}f_i(x_n)-\sum_{i\in F}\alpha_{s,i} f_i(y_n)}{d(x_n,y_n)} \right\vert
   + \frac{\varepsilon}{6M}\sum_{i\in I}\frac{\vert f_i(x_n)-f_i(y_n)\vert}{d(x_n,y_n)}\\
   & \leq \left\vert \sum_{n=1}^{n_0}\lambda_n \frac{\sum_{i\in F}\alpha_{s,i}f_i(x_n)-\sum_{i\in F}\alpha_{s,i} f_i(y_n)}{d(x_n,y_n)} \right\vert+\frac{\varepsilon}{6}.
\end{align*}
Putting all together, we get
$$\vert (\gamma\circ R)(\alpha)\vert<\frac{\varepsilon}{3}+\frac{\varepsilon}{6}+\left\vert \sum_{n=1}^{n_0}\lambda_n \frac{\sum_{i\in F}\alpha_{s,i}f_i(x_n)-\sum_{i\in F}\alpha_{s,i} f_i(y_n)}{d(x_n,y_n)} \right\vert.$$
Using once again that 
$$\sum_{i\notin F}\vert f_i(x_n)\vert+\sum_{i\notin F}\vert f_i(y_n)\vert<\frac{\varepsilon  d(x_n,y_n)}{6\lambda_n n_0},$$ 
we infer
\[
\begin{split}
&\left\vert \sum_{n=1}^{n_0}\lambda_n \frac{\sum_{i\in F}\alpha_{s,i}f_i(x_n)-\sum_{i\in F}\alpha_{s,i} f_i(y_n)}{d(x_n,y_n)} \right\vert\\
& \leq \left\vert \sum_{n=1}^{n_0}\lambda_n \frac{\sum_{i\in I}\alpha_{s,i}f_i(x_n)-\sum_{i\in I}\alpha_{s,i} f_i(y_n)}{d(x_n,y_n)} \right\vert
 +\left\vert \sum_{n=1}^{n_0}\lambda_n \frac{\sum_{i\notin F}\alpha_{s,i}f_i(x_n)-\sum_{i\notin F}\alpha_{s,i} f_i(y_n)}{d(x_n,y_n)} \right\vert\\
& <\left\vert \sum_{n=1}^{n_0}\lambda_n \frac{\sum_{i\in I}\alpha_{s,i}f_i(x_n)-\sum_{i\in I}\alpha_{s,i} f_i(y_n)}{d(x_n,y_n)} \right\vert
 +\sum_{n=1}^{n_0}\lambda_n \frac{\sum_{i\notin F}\vert f_i(x_n)\vert+\sum_{i\notin F}\vert f_i(y_n)\vert }{d(x_n,y_n)}\\
& < \left\vert \sum_{n=1}^{n_0}\lambda_n \frac{\sum_{i\in I}\alpha_{s,i}f_i(x_n)-\sum_{i\in I}\alpha_{s,i} f_i(y_n)}{d(x_n,y_n)} \right\vert+\frac{\varepsilon}{6}.
\end{split}\]
Thus
$$\vert (\gamma\circ R)(\alpha)\vert<\frac{\varepsilon}{3}+\frac{\varepsilon}{6}+\frac{\varepsilon}{6}+\left\vert \sum_{n=1}^{n_0}\lambda_n \frac{\sum_{i\in I}\alpha_{s,i}f_i(x_n)-\sum_{i\in I}\alpha_{s,i} f_i(y_n)}{d(x_n,y_n)} \right\vert.$$
Finally, observe that
\[
\begin{split}
\sum_{n=1}^{n_0}\lambda_n \frac{\sum_{i\in I}\alpha_{s,i}f_i(x_n)-\sum_{i\in I}\alpha_{s,i} f_i(y_n)}{d(x_n,y_n)}
& =\sum_{i=1}^{n_0}\lambda_n \frac{ R(\alpha_s)(x_n)- R(\alpha_s)(y_n)}{d(x_n,y_n)}\\
& = \left(\sum_{n=1}^{n_0}\lambda_n m_{x_n,y_n}\right)(R(\alpha_s))= \gamma_0(R(\alpha_s)).
\end{split}
\]
Thus
$$\vert (\gamma\circ R)(\alpha)\vert<\frac{2\varepsilon}{3}+\vert  \gamma_0(R(\alpha_s))\vert.$$
Using once again that $\Vert \gamma-\gamma_0\Vert<\varepsilon/6M$, we get that
\[
\begin{split}
\vert  \gamma_0(R(\alpha_s))\vert
&\leq \vert  \gamma(R(\alpha_s))\vert+\vert  (\gamma-\gamma_0)(R(\alpha_s))\vert\\
& \leq \vert  \gamma(R(\alpha_s))\vert+\Vert R\Vert \Vert \alpha_s\Vert_\infty\Vert \gamma-\gamma_0\Vert\\
& \leq \vert  \gamma(R(\alpha_s))\vert+\Vert  R\Vert \frac{\varepsilon}{6M}<\vert \gamma(R(\alpha_s))\vert+\frac{\varepsilon}{6}.
\end{split}
\]
Using finally that $\alpha_s\in \ker(\gamma\circ R)$, we get $ \gamma(R(\alpha_s))=0$. Consequently,
$$\vert (\gamma\circ R)(\alpha)\vert<\frac{2\varepsilon}{3}+\vert\gamma(R(\alpha_s))\vert+\frac{\varepsilon}{6}=\frac{5\varepsilon}{6}<\varepsilon.$$
Since $\varepsilon>0$ was arbitrary, we get $(\gamma\circ R)(\alpha)=0$, proving that $\ker(\gamma\circ R)\cap B_{\ell_\infty(I)}$ is $w^*$-closed. This finishes the proof.
\end{proof}

If $\left((x_i,y_i)\right)_{i\in I}$ is a Lipschitz interpolating set in $\widetilde{X}$ for $\Lip_0(X)$, the combination of Theorems \ref{rem:levanvarapouloset} and \ref{cor:caravarapoulos0} shows that the existence of a Beurling set of functions in $\Lipo(X)$ for $\left((x_i,y_i)\right)_{i\in I}$ characterizes the Lipschitz interpolation of such a set as linear. 

Now, the weak*-to-weak* continuity of the lifting $R:\ell_\infty(I)\to \Lipo(X)$ of the Lipschitz interpolating operator $T:\Lipo(X)\to \ell_\infty(I)$ associated to $\left((x_i,y_i)\right)_{i\in I}$ is the cornerstone for establishing new characterizations. Compare the following theorem to \cite[Theorem 5.2]{Muj-91}, a result of Mujica on interpolating sequences for spaces of bounded holomorphic maps between Banach spaces.

\begin{theorem}\label{cor:caravarapoulos}
Let $\left((x_i,y_i)\right)_{i\in I}$ be a Lipschitz interpolating set in $\widetilde{X}$ for $\Lipo(X)$ and let $M$ be its Lipschitz interpolating constant. The following statements are equivalent:
\begin{enumerate}
    \item There exists a Beurling set $(f_i)_{i\in I}$ of functions in $\Lipo(X)$ for $\left((x_i,y_i)\right)_{i\in I}$.
		\item There exists $R\in\L((\ell_\infty(I),w^*);(\Lipo(X),w^*))$ such that $\left\|R\right\|=M$ and $T\circ R=\mathrm{Id}_{\ell_\infty(I)}$.
    \item There exists $P\in\L(\F(X),\ell_1(I))$ such that $\Vert P\Vert=M$ and $P\circ S=\mathrm{Id}_{\ell_1(I)}$.
		\item There exists $f\in\Lipo(X,\ell_1)$ such that $\left\|f\right\|=M$ and $\left(\dfrac{f(x_i)-f(y_i)}{d(x_i,y_i)}\right)_{i\in I}=\left(e_i\right)_{i\in I}$.
    \item There exists $R_0\in\L(c_0(I),\Lipo(X))$ such that $\left\|R_0\right\|\leq M$ and $T\circ R_0=\Id_{c_0(I)}$.
\end{enumerate}
\end{theorem}

\begin{proof}
$(i) \Rightarrow (ii)$ follows from Theorems \ref{rem:levanvarapouloset} and \ref{theo:varapoadjointope}.

If $(ii)$ holds, then there exists $P\in\L(\F(X),\ell_1(I))$ such that $P^*=R$ and so $\left\|P\right\|=M$. Moreover, the equality 
$$
\left(\mathrm{Id_{\ell_1(I)}}\right)^*=\mathrm{Id}_{\ell_\infty(I)}=T\circ  R=S^*\circ P^*=(P\circ S)^*
$$
yields that $P\circ S=\Id_{\ell_1(I)}$, and this proves $(iii)$. 

If $(iii)$ holds, it is known that there exists a unique mapping $f\in\Lipo(X,\ell_1(I))$ such that $P=T_f$ and $\left\|f\right\|=||T_f||$. Moreover, 
$$
e_i=(P\circ S)(e_i)=T_f(m_{x_i,y_i})=m_{x_i,y_i}(f)=\frac{f(x_i)-f(y_i)}{d(x_i,y_i)}
$$ 
for each $i\in I$, and thus $(iv)$ is proved. 

Assume that $(iv)$ holds and let $f\in\Lipo(X,\ell_1(I))$ be such that $\left\|f\right\|=M$ and $\dfrac{f(x_j)-f(y_j)}{d(x_j,y_j)}=e_j$ for all $j\in I$. Denote $f=(f_i)_{i\in I}$ and note that  
$$
\left\|f\right\|=\sup_{(x,y)\in\widetilde{X}}\frac{\left\|f(x)-f(y)\right\|_1}{d(x,y)}=\sup_{(x,y)\in\widetilde{X}}\sum_{i\in I}\frac{\left|f_i(x)-f_i(y)\right|}{d(x,y)}.
$$
Consider $R_0\colon c_0(I)\to\Lipo(X)$ given by 
$$
R_0(\alpha)=\sum_{i\in I}\alpha_if_i\qquad (\alpha=(\alpha_i)_{i\in I}\in c_0(I)).
$$
We next prove that $R_0$ is a well-defined bounded linear operator with $\left\|R_0\right\|\leq \left\|f\right\|$. Indeed, let $\alpha=(\alpha_i)_{i\in I}\in c_0(I)$. Given $x\in X$ and $F\subseteq I$ finite, we have  
$$
\sum_{i\in F}\left|\alpha_i f_i(x)\right|
\leq \Vert \alpha\Vert_\infty\sum_{i\in F}\left|f_i(x)\right|
\leq \Vert \alpha\Vert_\infty\sum_{i\in I}\left|f_i(x)\right|
=\Vert \alpha\Vert_\infty\left\|f(x)\right\|_1,
$$
and therefore $R_0(\alpha)(x)\in\R$. Consider the function $R_0(\alpha)\colon X\to\mathbb R$. Clearly, $R_0(\alpha)(0)=0$. For any $(x,y)\in\widetilde{X}$, we get that 
$$
\frac{\vert R_0(\alpha)(x)-R_0(\alpha)(y)\vert}{d(x,y)}
=\left|\sum_{i\in I}\frac{\alpha_i(f_i(x)-f_i(y))}{d(x,y)}\right|
\leq \sum_{i\in I}\frac{\left|\alpha_i\right|\left|f_i(x)-f_i(y)\right|}{d(x,y)}
\leq \Vert \alpha\Vert_\infty M,   
$$
and so $R_0(\alpha)\in\Lipo(X)$ with $\Vert R_0(\alpha)\Vert\leq M$. In view of the linearity of $R_0$ we deduce that $R_0\colon c_0(I)\to\Lipo(X)$ is continuous with $\Vert R_0\Vert\leq M$. Moreover,  
$$
(T\circ R_0)(\alpha)
=T\left(\sum_{i\in I}\alpha_if_i\right)
=\left(\sum_{i\in I}\frac{\alpha_i(f_i(x_j)-f_i(y_j))}{d(x_j,y_j)}\right)_{j\in I}
=\left(\sum_{i\in I}\alpha_ie_j(i)\right)_{j\in I}
=\alpha
$$
for all $\alpha=(\alpha_j)_{j\in I}\in c_0(I)$, and so we have (v). 

Finally, to prove $(v)\Rightarrow (i)$, assume that there exists $R_0\in\L(c_0(I),\Lipo(X))$ such that $\left\|R_0\right\|\leq M$ and $T\circ R_0=\Id_{c_0(I)}$. Define $f_i=R_0(e_i)\in\Lipo(X)$ for every $i\in I$. An easy verification gives
$$
m_{x_j,y_j}(f_i)=\frac{f_i(x_j)-f_i(y_j)}{d(x_j,y_j)}=\frac{R_0(e_i)(x_j)-R_0(e_i)(y_j)}{d(x_j,y_j)}=T(R_0(e_i))(j)=e_i(j)=\delta_{ij}
$$
for all $i,j\in I$. On the other hand, the adjoint operator $(R_0)^*\colon\Lipo(X)^*\to c_0(I)^*$ comes given by $(R_0)^*(\phi)(e_i)=(\phi\circ R_0)(e_i)=\phi(f_i)$ for all $\phi\in\Lipo(X)^*$ and $i\in I$. Hence we have that 
$$
\sum_{i\in I}|\phi(f_i)|=\left\|(R_0)^*(\phi)\right\|_1\leq \left\|R_0\right\|\left\|\phi\right\|
$$
for all $\phi\in\Lipo(X)^*$. In particular, this implies that  
$$
\sum_{i\in I}\frac{\left|f_i(x)-f_i(y)\right|}{d(x,y)}=\sum_{i\in I}\left|m_{x,y}(f_i)\right|\leq \left\|R_0\right\|||m_{x,y}||\leq M
$$
for all $(x,y)\in\widetilde{X}$, and thus 
$$
\sup_{(x,y)\in\widetilde{X}}\sum_{i\in I}\frac{\left|f_i(x)-f_i(y)\right|}{d(x,y)}\leq M,
$$
as required. This proves $(i)$ and the proof of the theorem is complete. 
\end{proof}

Some comments on Theorem \ref{cor:caravarapoulos}:

\begin{remarks}\label{rem-to-cor:caravarapoulos}
\begin{enumerate}
\item Theorem \ref{cor:caravarapoulos} is also true if in the statements $(ii)$, $(iii)$ and $(iv)$, the respective norms of $R$, $P$ and $f$ are taken less or equal than $M$.
\item The operators $P\in\L(\F(X),\ell_1)$ and $R\in\L(\ell_\infty,\Lipo(X))$ coincide -- up to isometric isomorphisms -- with the linearisation $T_f$ and its adjoint $(T_f)^*$ of the function $f\in\Lipo(X,\ell_1)$, respectively.
\item Affirmations $(ii)$ and $(v)$ show that $\left((x_i,y_i)\right)_{i\in I}$ is linear Lipschitz interpolating and linear Lipschitz $c_0(I)$-interpolating for $\Lipo(X)$, respectively. This last terminology should be self-explanatory (see Definitions \ref{def-LIC-0} and \ref{def-LIC-00}). 
\end{enumerate}
\end{remarks}

Now, the following question is natural.

\begin{question}\label{quest:varapoulos}
Let $\left((x_i,y_i)\right)_{i\in I}$ be a Lipschitz interpolating set in $\widetilde{X}$ for $\Lipo(X)$. Is there always a Beurling set of functions $(f_i)_{i\in I}$ in $\Lipo(X)$ associated to $\left((x_i,y_i)\right)_{i\in I}$?
\end{question}

Taking into account Theorems \ref{theo:caraLipschitintergeneral} and \ref{cor:caravarapoulos}, an affirmative answer to Question~\ref{quest:varapoulos} can be read in the following terms: if $\left((x_i,y_i)\right)_{i\in I}$ is a set in $\widetilde{X}$ such that the operator $S\colon\ell_1(I)\to\F(X)$ is bounded below and $N:=\max\{J>0\colon J\Vert \lambda\Vert\leq \Vert S(\lambda)\Vert,\; \forall \lambda\in\ell_1(I)\}$, then automatically there exists an operator $P\in\L(\F(X),\ell_1(I))$ with $\Vert P\Vert =1/N$ such that $P\circ S=\mathrm{Id}_{\ell_1(I)}$. 

In a more informal language, the above can be simplified as follows: if $(m_{x_i,y_i})_{i\in I}$ is $N$-equivalent to the canonical basis of $\ell_1(I)$, then such a set of molecules is $N$-complemented in $\F(X)$. Observe that this problem has been recently analysed in \cite{OstOst-24} with an affirmative answer when $N=1$ and $\left((x_i,y_i)\right)_{i\in I}$ is a finite \cite[Lemma 2.1]{OstOst-24} or a countable set \cite[Theorem 1.3]{OstOst-24}.

Following those results, let us obtain a slight generalisation of \cite[Theorem 1.3]{OstOst-24}, but whose proof is a literal translation of the original proof to the setting of arbitrary sets.

\begin{theorem}\label{theo:ovstroskii}
Let $\left((x_i,y_i)\right)_{i\in I}$ be a Lipschitz interpolating set in $\widetilde{X}$ for $\Lipo(X)$ whose Lipschitz interpolating constant is $1$. Then there exists a Beurling set of functions in $\Lipo(X)$ for $\left((x_i,y_i)\right)_{i\in I}$.
\end{theorem}

\begin{proof}
Denote $\P_{\F}:=\{F\subseteq I\colon F\mbox{ is finite}\}$, which is a directed set with the classical order given by the inclusion. Let $F\in\P_{\F}$ be of cardinality $n$ with $n\in\N$. Since $\mathrm{span}\{m_{x_i,y_i}\colon i\in F\}$ is isometrically $\ell_1^n$, then \cite[Theorem 2.1]{OstOst-20} implies that $\{x_i,y_i\colon i\in F\}$ is a set of pairs forming a minimum-weight matching in $K(\{x_i,y_i\colon i\in F\})$ in the language of \cite[Lemma 2.1]{OstOst-24} (see also \cite[Corollary 2.7]{AliRZ-20} and the subsequent paragraph). Because of that, the proof of \cite[Lemma 2.1]{OstOst-24} implies that, for every $i\in F$, there exists $f_{i,F}\in S_{\Lipo(X)}$ such that $m_{x_j,y_j}(f_i)=\delta_{ij}$ for all $j\in F$ and the operator $P_F\colon\F(X)\to\mathrm{span}\{m_{x_i,y_i}\colon i\in F\}$, defined by $P_F(\gamma)=\sum_{i\in F}\gamma(f_{i,F})m_{x_i,y_i}$, is a surjective norm-one projection. Since $\mathrm{span}\{m_{x_i,y_i}: i\in F\}$ is isometrically $\ell_1^n$, we get
$$
1\geq \left\Vert \sum_{i\in F}m_{x,y}(f_{i,F})m_{x_i,y_i}\right\Vert=\sum_{i\in F}\vert m_{x,y}(f_{i,F})\vert
$$
holds for every $(x,y)\in \widetilde{X}$. Since $f_{i,F}\in B_{\Lipo(X)}$ for every $F\in\P_{\F}$ and $i\in F$, we can select, for any $i\in I$, a $w^*$-cluster point $f_i$ of the net $(f_{i,F})_{F\in S, F\supset\{i\}}$. 

In order to prove that $(f_i)_{i\in I}$ is a Beurling set in $\Lipo(X)$ associated to $\left((x_i,y_i)\right)_{i\in I}$, take $(x,y)\in \widetilde{X}$. Given any finite set $G\subseteq I$, select $F\in\P_{\F}$ with $G\subseteq F$. Now we get, by the properties of $f_{i,F}$, that
$$
\sum_{i\in G} \vert m_{x,y}(f_{i,F})\vert\leq \sum_{i\in F}\vert m_{x,y}(f_{i,F})\vert\leq 1.
$$
Taking into account that the above inequality holds true for every $F\geq G$ in $\P_{\F}$ and that $f_i$ is a $w^*$-cluster point of the net $(f_{i,F})_{F\in S, F\supset\{i\}}$ for every $i\in I$, we infer that
$$
\sum_{i\in G}\vert m_{x,y}(f_{i})\vert\leq 1.
$$
Now, since the above inequality is satisfied for every finite set $G\subseteq I$, then, by definition, there exists the following sum
$$
\sum_{i\in I}\vert m_{x,y}(f_{i})\vert\leq 1.
$$
Since $(x,y)\in \widetilde{X}$ was arbitrary, we conclude that
$$
\sup\limits_{(x,y)\in \widetilde{X}} \sum_{i\in I}\frac{\vert f_i(x)-f_i(y)\vert}{d(x,y)}
=\sup\limits_{(x,y)\in \widetilde{X}} \sum_{i\in I}\vert m_{x,y}(f_i)\vert\leq 1.
$$
According to Definition~\ref{def:Berurlingeq} we get that $(f_i)_{i\in I}$ is a Beurling set of functions in $\Lipo(X)$ associated to $\left((x_i,y_i)\right)_{i\in I}$, as requested.
\end{proof}

We now give two more examples where Question~\ref{quest:varapoulos} has an affirmative answer. In the former example, we will exploit the well known property that $\ell_\infty^2$ and $\ell_1^2$ are isometrically isomorphic Banach spaces and the fact that $\ell_1$-spaces have the lifting property.

\begin{example}
Let $X$ be a pointed metric space and let $\left\{(x_1,y_1), (x_2,y_2)\right\}$ be an interpolating set for $\Lipo(X)$ in $\widetilde{X}$. Let $T:\Lipo(X)\to \ell_\infty^2$ be its associate interpolating operator with Lipschitz interpolating constant $M$. We claim that there exists a Beurling set $\{f_1,f_2\}$ in $\Lipo(X)$ associated to $\left\{(x_1,y_1), (x_2,y_2)\right\}$. 

Indeed, consider the onto linear isometry $q:\ell_1^2\to \ell_\infty^2$ defined by $q(x,y)=(x+y,x-y)$. Since $T$ is onto, for $i=1,2$ there exists $f_i\in \Lipo(X)$ so that $T(f_i)=q(e_i)$ and $\Vert f_i\Vert\leq M$. Now define $\varphi: \ell_1^2\to \Lipo(X)$ by $\varphi(a,b)=a f_1+b f_2$. Observe that $\varphi\in\L(\ell_1^2,\Lipo(X))$ with $\Vert \varphi\Vert\leq M$, and it holds that $T\circ\varphi=q$. So $R:=\varphi\circ q^{-1}\in\L(\ell_\infty^2,\Lipo(X))$ with $\Vert  R\Vert\leq M$ and $T\circ R=\Id_{\ell_\infty^2}$, and the proof is finished by Theorem \ref{cor:caravarapoulos0}.
\end{example}

\begin{example}
Let $X$ be a finite set and consider $D:=\left((x_i,y_i)\right)_{i\in I}\subseteq \widetilde{X}$. Assume that $(X,D)$ defines a connected graph. The following are equivalent:
\begin{enumerate}
\item $(X,D)$ contains no cycle.
\item The Lipschitz interpolating operator $T:\Lipo(X)\to \ell_\infty(I)$ associated to $\left((x_i,y_i)\right)_{i\in I}$ is onto.
\item $\left((x_i,y_i)\right)_{i\in I}$ has a finite Beurling set of functions in $\Lipo(X)$.
\end{enumerate}

$(ii)\Rightarrow(i)$: Assume that $(M,D)$ contains a cycle. Then there exists $i_1,\ldots, i_n\subseteq I$ such that $x_i=y_{i+1}$ for every $1,\ldots,n-1$ and $x_n=y_1$. Then there is no function $f\in\Lipo(X)$ such that $m_{x_1,y_1}(f)=1$ and $m_{x_i,y_i}(f)=0$ for every $2\leq i\leq n$.  Indeed, if such a function existed, then $f(x_1)=f(y_1)+d(x_1,y_1)$. Moreover, since $f(x_i)=f(y_i)$ for every $2\leq i\leq n$, then we would get $f(x_i)=f(y_1)+d(x_1,y_1)$. This would imply $f(x_n)=f(y_1)+d(x_1,y_1)$. But $x_n=y_1$ would imply $f(y_1)=f(y_1)+d(x_1,y_1)$, which is impossible. 

$(iii)\Rightarrow(ii)$ is obvious. Let us now prove $(i)\Rightarrow(iii)$. Assume that $(M,D)$ is a connected graph containing no cycle. By the assumption we can assume that $(x_1,y_1),\ldots, (x_n,y_n)$ is a simple path and, up to a relabeling, we can assume with no loss of generality that $y_1=0$. 

We claim that in the above situation $T$ is a bijection. Indeed, given $\alpha=(\alpha_1,\ldots, \alpha_n)\in \ell_\infty^n$, we claim that there exists a unique $f\in \Lipo(X)$ such that $T(f)=\alpha$. Indeed, given $x\in X$ there exists a unique subset $i_1,\ldots, i_k\subseteq I$ such that $y_{i_1}=0$ and $x_{i_k}=x$. Now the condition $f(x_{i_j})=f(y_{i_j}) +\alpha_{i_j}d(x_{i_j},y_{i_j})$ holds for every $1\leq j\leq k$ and $f(y_{i_1})=0$ implies $f(x_{i_k})$ is uniquely defined. Moreover $f$ is well defined since the path joining $0$ and $x$ is unique. 

Since $T:\Lipo(X)\to \ell_\infty^n$ is bijective and continuous, its inverse $ R:\ell_\infty^n\to \Lipo(X)$ is continuous too, and it is not difficult to prove that $\Vert  R\Vert=M$. Taking $f_i= R(e_i)$ for $i=1,\ldots,n$, it is immediate that $(f_i)_{i=1}^n$ is a Beurling set in $\Lipo(X)$ for $\left((x_i,y_i)\right)_{i\in I}$, and the proof is finished.
\end{example}

We now study the connection between the Lipschitz interpolation of a set $\left((x_i,y_i)\right)_{i\in I}$ in $\widetilde{X}$ and the separation which provides the following distance. 

\begin{definition}
Let $X$ be a pointed metric space. Define the Lipschitz-molecular metric (L-molecular metric, for short) on $\widetilde{X}$ by 
$$
\rho((x,y),(u,v))=\left\|m_{x,y}-m_{u,v}\right\|
$$
for any $(x,y),(u,v)\in\widetilde{X}$. For $(x,y)\in\widetilde{X}$ and $0<r\leq 2$, define
$$
D_\rho((x,y),r)=\left\{(u,v)\in\widetilde{X}\colon \rho((x,y),(u,v))<r\right\}.
$$
A set $\left((x_i,y_i)\right)_{i\in I}$ in $\widetilde{X}$ is said to be L-molecularly $r$-separated if 
$$
\rho((x_i,y_i),(x_j,y_j))>r\qquad (i,j\in I,\; i\neq j),
$$
and the constant of separation of $\left((x_i,y_i)\right)_{i\in I}$ is defined as 
$$
\inf\left\{\rho((x_i,y_i),(x_j,y_j))\colon i,i\in ,\; i\neq j\right\}.
$$ 
\end{definition}

L-molecularly $r$-separated sets $\left((x_i,y_i)\right)_{i\in I}\subseteq\widetilde{X}$ are characterized in \cite[Lemma 3.12]{CasALL-20} under the name of uniformly discrete set of molecules as follows. 

\begin{lemma}\label{lemaCascales}\cite[Lemmas 1.3 and 3.12]{CasALL-20}. 
Let $(x,y),(u,v)\in\widetilde{X}$. Then
$$
\rho((x,y),(u,v))\leq 2\frac{d(x,u)+d(y,v)}{\max\{d(x,y),d(u,v)\}}.
$$
If, moreover, $\rho((x,y),(u,v))<1$, then 
$$
\frac{\max\{d(x,u),d(y,v)\}}{\min\{d(x,y),d(u,v)\}}\leq \rho((x,y),(u,v)).
$$
As an application, a set $\left((x_i,y_i)\right)_{i\in I}$ in $\widetilde{X}$ is L-molecularly $r$-separated if and only if $(r/2)d(x_i,y_i)\leq d(x_i,x_j)+d(y_i,y_j)$ for all $i,j\in I$ with $i\neq j$. $\hfill$ $\Box$
\end{lemma}

Notice that a set $\left((x_i,y_i)\right)_{i\in I}$ in $\widetilde{X}$ is L-molecularly $2r$-separated if and only if $D_\rho((x_i,y_i),r)\cap D_\rho((x_j,y_j),r)=\emptyset$ whenever $i,j\in I$ with $i\neq j$.

We now show that Lipschitz interpolating sets for $\Lipo(X)$ are necessarily separated in the following sense. Some similar results were stated with interpolating sequences for Bloch functions, weighted analytic functions and uniform algebras with respect to the pseudo-hyperbolic metric (see \cite{Ate-92,BlaGalLinMir-19,GalGamLin-04,GalLinMir-09}).

\begin{proposition}\label{prop-2}
Let $\left((x_i,y_i)\right)_{i\in I}$ be a Lipschitz interpolating set for $\Lipo(X)$ in $\widetilde{X}$ with Lipschitz interpolation constant $M$. Then $\left((x_i,y_i)\right)_{i\in I}$ is L-molecularly $1/M$-separated.
\end{proposition}

\begin{proof}
Let $i,j\in I$ be with $i\neq j$. We can find a $f\in\Lipo(X)$ with $\left\|f\right\|\leq M$ such that
$$
\frac{f(x_i)-f(y_i)}{d(x_i,y_i)}=1\quad \text{and}\quad \frac{f(x_j)-f(y_j)}{d(x_j,y_j)}=0.
$$ 
It follows that
\begin{align*}
\rho((x_i,y_i),(x_j,y_j))&\geq \frac{1}{M}\left|\delta_{(x_i,y_i)}(f)-\delta_{(x_j,y_j)}(f)\right|\\
&=\frac{1}{M}\left|\frac{f(x_i)-f(y_i)}{d(x_i,y_i)}-\frac{f(x_j)-f(y_j)}{d(x_j,y_j)}\right|=\frac{1}{M}
\end{align*}
for all $i,j\in I$ with $i\neq j$. Hence $\left((x_i,y_i)\right)_{i\in I}$ is L-molecularly $1/M$-separated.
\end{proof}

We finish this section with the study of the stability of Lipschitz interpolating sets for $\Lipo(X)$ in $\widetilde{X}$. See \cite[Proposition 7]{Ate-92} for a similar result in the study of interpolating sequences for the derivatives of Bloch functions. 

\begin{proposition}\label{prop-3}
Let $\left((x_i,y_i)\right)_{i\in I}$ be a Lipschitz interpolating set for $\Lipo(X)$ in $\widetilde{X}$. Then there exists a number $\delta>0$, depending on $\left((x_i,y_i)\right)_{i\in I}$, such that if $\left((u_i,v_i)\right)_{i\in I}$ is a set in $\widetilde{X}$ satisfying that $\rho((x_i,y_i),(u_i,v_i))<\delta$ for all $i\in I$, then $\left((u_i,v_i)\right)_{i\in I}$ is Lipschitz interpolating for $\Lipo(X)$.
\end{proposition}

\begin{proof}
Since the set of all surjective operators between two Banach spaces $E$ and $F$ is open in $\L(E,F)$ by \cite[Chapter XV, Theorem 3.4]{Lan-93} and the Lipschitz interpolating operator $T\colon\Lipo(X)\to\ell_\infty(I)$ associated to $\left((x_i,y_i)\right)_{i\in I}$ for $\Lipo(X)$ is surjective, we can find a number $\delta=\delta(T)>0$ such that $T_0$ is surjective whenever $T_0\in\L(\Lipo(X),\ell_\infty(I))$ and $\left\|T-T_0\right\|\leq\delta$. 

Let $\left((u_i,v_i)\right)_{i\in I}$ be a set in $\widetilde{X}$ such that $\rho((x_i,y_i),(u_i,v_i))<\delta$ for all $i\in I$ and let $T_0\colon\Lipo(X)\to\ell_\infty(I)$ be the Lipschitz interpolating operator associated to $\left((u_i,v_i)\right)_{i\in I}$ for $\Lipo(X)$. Clearly, 
$$
\left\|T-T_0\right\|=\sup_{i\in I}\rho((x_i,y_i),(u_i,v_i))\leq\delta,
$$
and so $T_0$ is surjective.
\end{proof}


\section{Lipschitz interpolating sequences on compact metric spaces}\label{sect:compact}

Let $X$ be a pointed metric space and let $((x_n,y_n))_{n=1}^\infty$ be a Lipschitz interpolating sequence in $\widetilde{X}$. As we have exposed in Proposition \ref{prop-new}, the associate Lipschitz interpolating operator $T:\Lipo(X)\to\ell_\infty$ is the adjoint operator of $S\in\L(\ell_1,\F(X))$ given by
$$
S(\lambda):=\sum_{n=1}^\infty \lambda_n m_{x_n,y_n}\qquad (\lambda=(\lambda_n)_{n=1}^\infty\in\ell_1).
$$
There are some cases in which the Lipschitz free space $\F(X)$ is itself a dual space, say $\mathcal F(X)\cong Z^*$ for some closed linear subspace $Z\subseteq \Lipo(X)$. In this context, it would be natural to wonder when $S$ is an adjoint operator itself, say $S=R^*$, where $R:Z\to c_0$ is a surjective operator. In such case, we would obtain that $R$ would be, in an informal language, a \textit{$c_0$ interpolating operator for $Z$}. Observe that such ideas were previously considered on interpolating sequences for Bloch type spaces (see \cite[Theorem 2.4]{MirMal-22}).

Since one of the most classical setting in which the existence of a predual of $\mathcal F(X)$ is found is in the case that $X$ is a compact metric space, we will focus first on studying which property a sequence in $\widetilde{X}$ must satisfy to be a Lipschitz interpolating sequence. Note that in the context of compact pointed metric spaces (in particular, they are separable spaces), we may restrict ourselves to work with sequences $\left((x_n,y_n)\right)_{n=1}^\infty$ in $\widetilde{X}$.

We now give a necessary condition for a sequence $\left((x_n,y_n)\right)_{n=1}^\infty$ in $\widetilde{X}$ to be Lipschitz interpolating for $\Lipo(X)$. 

\begin{proposition}\label{prop:dist0interseqcompact}
Let $X$ be a compact pointed metric space and let $\left((x_n,y_n)\right)_{n=1}^\infty$ be a Lipschitz interpolating sequence for $\Lipo(X)$ in $\widetilde{X}$. Then $\left(d(x_n,y_n)\right)_{n=1}^\infty\to 0$ as $n\to\infty$.
\end{proposition}

\begin{proof}
Let $M$ be the Lipschitz interpolation constant of $\left((x_n,y_n)\right)_{n=1}^\infty$. Given $n\neq m$, Proposition \ref{prop-2} and Lemma \ref{lemaCascales} yield
$$
d(x_n,y_n)\leq 2M\left(d(x_n,x_m)+d(y_n,y_m)\right).
$$
In particular, 
$$
d(x_n,y_n)\leq 2M\left(d(x_n,x_{n+h})+d(y_n,y_{n+h})\right),
$$
for all $n,h\in\mathbb{N}$. Now, if $\left(d(x_n,y_n)\right)_{n=1}^\infty\to 0$ did not converge to $0$ as $n\to\infty$, we could find some $\varepsilon_0$ and two subsequences $(x_{n_k})_{k=1}^\infty$, $(y_{n_k})_{k=1}^\infty$ such that $d(x_{n_k},y_{n_k})\geq \varepsilon_0$ for all $k\in\N$. By the above we would get that 
$$
\varepsilon_0\leq 2M \left(d({x_{n_k},x_{n_k+h}})+ d({y_{n_k},y_{n_k+h}})\right)
$$
would hold for every $k, h\in\mathbb N$. Now, up to taking a further subsequence, we can assume that both $(x_{n_k})_{k=1}^\infty$ and $(y_{n_k})_{k=1}^\infty$ are convergent and, consequently, Cauchy sequences. Taking, by the Cauchy condition, $k\in\mathbb N$ large enough such that $d(x_{n_k},x_{n_{k+h}})<\varepsilon_0/4M$ and $d(y_{n_k},y_{n_{k+h}})<\varepsilon_0/4M$ holds for every $h\geq k$, we get
$$
\varepsilon_0\leq 2M \left(d({x_{n_k},x_{n_{k+h}}})+ d({y_{n_k},y_{n_{k+h}}})\right)<\varepsilon_0,
$$
a contradiction. Consequently, $\left(d(x_n,y_n)\right)_{n=1}^\infty\to 0$ as $n\to\infty$.
\end{proof}

At this point a natural question can be raised.

\begin{question}\label{question now}
Given a compact pointed metric space $X$ and a L-molecularly separated sequence of molecules $(m_{x_n,y_n})_{n=1}^\infty$ with $(x_n,y_n)\in\widetilde{X}$ for all $n\in\N$, is it true that $\left((x_n,y_n)\right)_{n=1}^\infty$ is Lipschitz interpolating for $\Lipo(X)$? 
\end{question}

An affirmative answer could be interpreted as a version for Lipschitz interpolating sequences of some results as, among others, \cite[Corollary 6]{Ate-92} established for Bloch functions, \cite[Proposition 3.3]{GalLinMir-09} for uniform algebras, or \cite[Corollary 5.3]{BlaGalLinMir-19} for weighted analytic functions.

In general, the above question has a negative answer, but a converse can be stated if we admit taking a subsequence. First we present the following example.

\begin{example}\label{exam:compactnointer}
For each $n\in\N$, let $X_n:=\{0,x_n,y_n\}$ be equipped with distances $d(0,x_n)=d(0,y_n)=d(x_n,y_n)=1/n$. Define $X:=\bigcup\limits_{n\in\mathbb N} X_n$ and declare the distance such that
$$
d(x,y)=d(x,0)+d(y,0)\qquad (x\in X_n,\; y\in X_k,\; n\neq k).
$$
This metric space is realizable, for instance, in $\left(\oplus_{n=1}^\infty \ell_\infty^2 \right)_1$. Consider the set of points in $\widetilde{X}$:
$$
Y=\left\{(x_n,0),(y_n,x_n),(0,y_n)\colon n\in\N\right\}.
$$
This sequence of points is not Lipschitz interpolating for $\Lipo(X)$ because it contains cycles (it is impossible that a function $f\in\Lipo(X)$ satisfies $m_{x_n,0}(f)=1$ and $m_{x_n,y_n}(f)=m_{0,y_n}(f)=0$ holds for every $n\in\N$).

We claim, however, that $\Vert m_{x,y}-m_{u,v}\Vert=2$ whenever $(x,y),(u,v)\in Y$. In order to do so, for every $n\in\mathbb N$, define the functions $f_n,g_n\colon X\to\mathbb R$ by $f(x_n)=1/n$ and $0$ otherwise, and, similarly, $g_n(y_n)=1/n$ and $g_n=0$ otherwise. It is not difficult to prove that $\Vert f_n\Vert=\Vert g_n\Vert=1$. Moreover, $(m_{x_n,0}-m_{y_n,x_n})(f_n)=2$ and $(m_{y_n,x_n}-m_{0,y_n})(g_n)=2$ for every $n\in\mathbb N$, and so $\Vert m_{x_n,0}-m_{y_n,x_n}\Vert=\Vert m_{y_n,x_n}-m_{0,y_n}\Vert=2$. 

Moreover, if $n\neq m$, observe that the defined distance implies that 
$$
\Vert f_n\pm f_m\Vert=\Vert f_n\pm g_m\Vert=\Vert g_n\pm g_m\Vert=1.
$$
Thus,
$$\Vert m_{x_n,0}-m_{x_m,0}\Vert\geq (m_{x_n,0}-m_{x_m,0})(f_n-f_m)=2.$$
Similarly,
$$\Vert m_{x_n,0}-m_{y_m,x_m}\Vert\geq (m_{x_n,0}-m_{y_m,x_m})(f_n-g_m)=2$$
and 
$$\Vert m_{x_n,0}-m_{0,y_m}\Vert\geq (m_{x_n,0}-m_{0,y_m})(f_n+g_m)=2.$$
Moreover,
$$\Vert m_{y_n,x_n}-m_{y_m,x_m}\Vert\geq (m_{y_n,x_n}-m_{y_m,x_m})(g_n+f_m)=2$$
and
$$\Vert m_{y_n,x_n}-m_{0,y_m}\Vert\geq (m_{y_n,x_n}-m_{0,y_m})(g_n+g_m)=2,$$
and, finally,
$$\Vert m_{y_n,0}-m_{y_m,0}\Vert\geq (m_{y_n,0}-m_{y_m,0})(-g_n+g_m)=2.$$
All the considered cases finish the proof.
\end{example}

With the aid of Lemma 2.6 in \cite{jr22}, we may establish the announced result. Observe that the compactness of $X$ can be removed.

\begin{theorem}\label{theo:suceLorto}
Let $X$ be a pointed metric space and let $((x_n,y_n))_{n=1}^\infty$ be a sequence in $\widetilde{X}$ such that $(d(x_n,y_n)))_{n=1}^\infty\rightarrow 0$ as $n\to\infty$. Then, for every $\varepsilon>0$,  there exists a subsequence $((x_{n_k},y_{n_k}))_{k=1}^\infty$ which is Lipschitz interpolating for $\Lipo(X)$ and its Lipschitz interpolating constant is smaller or equal than $1/(1-\varepsilon)$.
\end{theorem}

\begin{proof}
Choose a sequence of positive numbers $(\varepsilon_n)_{n=1}^\infty$ small enough to get $\prod\limits_{k=1}^\infty (1-\varepsilon_k)>1-\varepsilon$. Let us construct by induction a subsequence $(m_{x_{n_k},y_{n_k}})_{k=1}^\infty$ such that, given $p\in\mathbb N$, it holds that 
$$
\left\Vert \sum_{k=1}^p\lambda_k m_{x_{n_k},y_{n_k}} \right\Vert\geq \left( \prod\limits_{k=1}^p 1-\varepsilon_k\right)\sum_{k=1}^p\vert\lambda_k\vert
$$
for every $\lambda_1,\ldots, \lambda_p\in\mathbb R$. To do so, select $n_1=1$ for which clearly the above condition holds, assume that $n_1,\ldots, n_p$ have already been constructed, and let us construct $n_{p+1}$. 

Clearly, it follows that $(d(x_k,y_k))_{k=n_p+1}^\infty\rightarrow 0$ as $k\to\infty$. By \cite[Lemma 2.6]{jr22} we get that $\Vert\gamma+m_{x_k,y_k}\Vert_{k=n_p+1}^\infty\to 1+\Vert \gamma\Vert$ holds as $k\to\infty$ for every $\gamma\in\F(X)$, so $(m_{x_k,y_k})_{k=n_p+1}^\infty$ is an $L$-orthogonal sequence in the language of \cite{amr23}. By \cite[Lemma 3.1]{amr23} there exists $n_{p+1}$ large enough to satisfy that
$$
\Vert \gamma+\lambda m_{x_{n_{p+1}},y_{n_{p+1}}}\Vert\geq (1-\varepsilon_{p+1})(\Vert \gamma\Vert+\vert\lambda\vert)
$$
holds for all $\lambda\in\mathbb R$ and $\gamma\in\mathrm{span}\left\{m_{x_{n_i},y_{n_i}}\colon 1\leq i\leq p\right\}$. Given $\lambda_1,\ldots,\lambda_p,\lambda_{p+1}\in\mathbb R$, we get
\begin{align*}
\left\Vert \sum_{i=1}^{p+1}\lambda_i m_{x_{n_i},y_{n_i}} \right\Vert
& =\left\Vert \sum_{i=1}^{p}\lambda_i m_{x_{n_i},y_{n_i}} +\lambda_{p+1}m_{x_{n_{p+1}},y_{n_{p+1}}}\right\Vert\geq (1-\varepsilon_{p+1})\left(\left\Vert \sum_{i=1}^{p}\lambda_i m_{x_{n_i},y_{n_i}}\right\Vert +\vert\lambda_{p+1}\vert\right)\\
& \geq (1-\varepsilon_{p+1})\left(\left(\prod\limits_{i=1}^p 1-\varepsilon_i\right) \sum_{i=1}^p\vert\lambda_i\vert +\vert\lambda_{p+1}\vert\right)\geq \left(\prod\limits_{i=1}^{p+1}1-\varepsilon_i\right)\sum_{i=1}^{p+1}\vert\lambda_i\vert.
\end{align*}
Finally, given $\lambda=(\lambda_n)_{n=1}^\infty\in\ell_1$, it is clear by the inductive construction that
$$
\left\Vert\sum_{k=1}^\infty \lambda_k m_{x_{n_k},y_{n_k}} \right\Vert\geq \left( \prod\limits_{k=1}^\infty 1-\varepsilon_k\right)\sum_{k=1}^\infty \vert \lambda_k\vert\geq (1-\varepsilon)\Vert \lambda\Vert_{\ell_1},
$$
and this proves that the mapping $S:\ell_1\to \mathcal F(X)$, defined by $S(e_i):=m_{x_{n_i},y_{n_i}}$, is an into isomorphism satisfying that $(1-\varepsilon)\Vert\lambda\Vert_1\leq\Vert S(\lambda)\Vert$. Then the result follows by applying Proposition~\ref{prop:caralipsconstant}.
\end{proof}

We now proceed to study Lipschitz interpolation of sequences for little Lipschitz functions. If $\left((x_n,y_n)\right)_{n=1}^\infty$ is a sequence in $\widetilde{X}$ such that $\left(d(x_n,y_n)\right)_{n=1}^\infty\to 0$ as $n\to\infty$, then the restriction to $\lipo(X)$ of the Lipschitz interpolating operator $T\colon\Lipo(X)\to\ell_\infty$ maps $\lipo(X)$ to $c_0$ because 
$$
\lim_{n\to\infty}\frac{f(x_n)-f(y_n)}{d(x_n,y_n)}=0.
$$
This justifies the following notion introduced in a more general environment.

\begin{definition}\label{def-LIC-00}
Let $X$ be a pointed metric space. A set $\left((x_i,y_i)\right)_{i\in I}$ in $\widetilde{X}$ is Lipschitz $c_0(I)$-interpolating for $\lipo(X)$ if for any set $\left(\alpha_i\right)_{i\in I}\in c_0(I)$, there exists $f\in\lipo(X)$ such that 
$$
\frac{f(x_i)-f(y_i)}{d(x_i,y_i)}=\alpha_i
$$ 
for all $i\in I$.  
\end{definition}

Now it is time to go back to our original compact setting of this section. To do so, let $X$ be a compact pointed metric space and assume that $\lipo(X)$ separates the points of $X$ uniformly. In this case, $\mathcal F(X)^*\cong\lipo(X)$.

The following theorem is a version for Lipschitz spaces of a result on interpolating sequences for Bloch type spaces (see \cite[Theorem 2.4]{MirMal-22}).

\begin{theorem}\label{teo-1}
Let $X$ be a compact pointed metric space such that $\lipo(X)$ separates points uniformly, let $\left((x_n,y_n)\right)_{n=1}^\infty$ be a sequence of distinct points in $\widetilde{X}$ and let $M\geq 1$. The following statements are equivalent:
\begin{enumerate}
	\item $\left((x_n,y_n)\right)_{n=1}^\infty$ is Lipschitz interpolating for $\Lipo(X)$ and its Lipschitz interpolating constant is $M$.
    \item The operator $S:\ell_1\to \F(X)$ defined by $S(e_n):=m_{x_n,y_n}$ is an into isomorphism and 
    $$\frac{1}{M}=\max\{J>0\colon J\Vert\lambda\Vert\leq \Vert S(\lambda)\Vert,\; \forall\lambda\in\ell_1\}.$$
	\item $\left((x_n,y_n)\right)_{n=1}^\infty$ is Lipschitz $c_0$-interpolating for $\lipo(X)$ and its Lipschitz interpolating constant is $M$.
 	\item $M$ is the infimal constant satisfying the condition: for every $\alpha\in c_0$, there exists a function $f\in\Lipo(X)$ such that $T(f)=\alpha$ and $\Vert f\Vert\leq M$.
\end{enumerate}
\end{theorem}

\begin{proof} 
$(i)\Leftrightarrow(ii)$ has been shown in Theorem~\ref{theo:caraLipschitintergeneral}.

$(ii)\Rightarrow(iii)$: Assume that $(ii)$ holds. We claim that $S$ is weak*-to-weak* continuous, then $S=(R_1)^*$ for some $R_1\in\L(\lipo(X),c_0)$, and $(iii)$ will be a direct consequence of Proposition~\ref{prop:caralipsconstantdual}.

In order to prove our claim, let $\lambda=(\lambda_n)_{n=1}^\infty\in\ell_1$ and since $\ell_1$ is separable, take a sequence $(\lambda_k)_{k=1}^\infty$, with $\lambda_k=(\lambda_{k,n})_{n=1}^\infty\in\ell_1$ for every $k\in\N$, that converges to $\lambda$ as $k\to\infty$ in the weak* topology of $\ell_1$. Let us prove that $S((\lambda_k)_{k=1}^\infty)\to S(\lambda)$ as $k\to\infty$ in the weak* topology of $\F(X)$. It suffices to show that for every $f\in\lipo(X)\setminus\{0\}$, it holds that
$$
S((\lambda_k)_{k=1}^\infty)(f)=\left(\sum_{n=1}^\infty \lambda_{k,n} m_{x_n,y_n}(f)\right)_{k=1}^\infty\mathop{\to}\limits^{k\rightarrow \infty} \sum_{n=1}^\infty \lambda_n m_{x_n,y_n}(f)=S(\lambda)(f).
$$
To this end, select $\varepsilon>0$ and take $L>0$ such that $\Vert\lambda\Vert_1+\sup_{k\in\mathbb N}\Vert \lambda_k\Vert_1\leq L$ (since weak* convergent sequences are bounded). On a hand, since $((x_n,y_n))_{n=1}^\infty$ is Lipschitz interpolating for $\Lipo(X)$ and $X$ is compact, we get $(d(x_n,y_n))_{n=1}^\infty\to 0$ as $n\to\infty$ in virtue of Proposition~\ref{prop:dist0interseqcompact}. Since $f\in\lipo(X)$, we get that $(m_{x_n,y_n}(f))_{n=1}^\infty\to 0$ as $n\to\infty$. Thus one may select $n_0\in\mathbb N$ such that $\vert m_{x_n,y_n}(f)\vert<\varepsilon/2L$ for every $n\geq n_0$. 

On the other hand, since $(\lambda_k)_{k=1}^\infty\to\lambda$ as $k\to\infty$ in the weak* topology of $\ell_1$, we infer that, for each $n\in\mathbb N$, the real sequence $(\lambda_{k,n})_{k=1}^\infty\to\lambda_n$ as $k\to\infty$. So we can find $k_0\in\N$ such that if $k\geq k_0$, then 
$$
\vert \lambda_{k,n}-\lambda_n\vert<\frac{\varepsilon}{2n_0\Vert f\Vert} \quad (1\leq n\leq n_0).
$$
Note that $\vert (S((\lambda_k)_{k=1}^\infty)-S(\lambda))(f)\vert<\varepsilon$ whenever $k\geq k_0$, since 
\begin{align*}
\vert (S((\lambda_k)_{k=1}^\infty)-S(\lambda))(f)\vert
&=\left\vert \sum_{n=1}^\infty (\lambda_{k,n}-\lambda_n)m_{x_n,y_n}(f) \right\vert\\
& \leq \sum_{n=1}^{n_0}\vert \lambda_{k,n}-\lambda_n\vert \vert m_{x_n,y_n}(f)\vert +\sum_{n=n_0+1}^\infty \vert \lambda_{k,n}-\lambda_n\vert\vert m_{x_n,y_n}(f)\vert\\
& < \sum_{n=1}^{n_0}\frac{\varepsilon}{2n_0}\frac{\vert m_{x_n,y_n}(f)\vert}{\Vert f\Vert}+\sum_{n=n_0+1}^\infty (\vert \lambda_{k,n}\vert +\vert \lambda_n\vert)\frac{\varepsilon}{2L}\\
& <\frac{\varepsilon}{2}+\frac{\varepsilon}{2}\frac{\sum_{n=1}^\infty \vert \lambda_{k,n}\vert +\sum_{n=1}^\infty \vert \lambda_n\vert}{L}<\varepsilon.
\end{align*}
This proves that $S$ is weak*-to-weak* continuous, as desired.

$(iii)\Rightarrow(iv)$ is obvious. 

$(iv)\Rightarrow (i)$: Let $\alpha\in\ell_\infty$. By Goldstein'n Theorem and by the separability of $\ell_1$, there is a sequence $\left(\beta_k\right)_{k=1}^\infty$ in $\Vert \alpha\Vert B_{c_0}$ such that $\left(\beta_k\right)_{k=1}^\infty\to\alpha$ as $k\to\infty$ in the weak* topology. By $(iv)$, for any $k\in\N$, there exists $f_k\in\Lipo(X)$ with $\left\|f_k\right\|\leq \left(M+\frac{1}{k}\right)\left\|\beta_k\right\|_\infty$ such that $T(f_k)=\beta_k$. Since $\F(X)$ is separable, Banach--Alaoglu Theorem provides a subsequence $\left(f_{k_m}\right)_{m=1}^\infty$ of $\left(f_k\right)_{k=1}^\infty$ which w*-converges to some $f\in\Lip_0(X)$. Since $(T|_{\lip_0(X)})^{**}=T$ (proved in $(ii)\Rightarrow(iii)$), we get that $T$ is weak*-to-weak* continuous, so 
$$
\alpha=w^*-\lim_{m\to\infty}\beta_{k_m}=w^*-\lim_{m\to\infty}T(f_{k_m})=T(f).
$$ 
This proves that $((x_n,y_n))_{n=1}^\infty$ is a Lipschitz interpolating sequence for $\Lipo(X)$. In order to compute its Lipschitz interpolating constant, observe that the $w^*$-lower semicontinuity of the norm of $\Lipo(X)$ implies
$$
\Vert f\Vert\leq \liminf_{k}\Vert f_k\Vert\leq M \Vert \alpha\Vert.
$$
This also proves that $M$ is the Lipschitz interpolating constant, as desired.
\end{proof}

\section{Lipschitz interpolating sets for vector-valued Lipschitz functions}\label{sect:vectorval}

In this section we will have a look to the case of vector-valued Lipschitz functions. In the study of interpolating sequences, the extension of the scalar-valued case to the vector-valued one also has been addressed in the setting of bounded holomorphic functions (see \cite[Theorem 5.2]{Muj-91} and  \cite[p. 286]{Gar-81}). 

\begin{definition}\label{def:Lipschitvect}
Let $X$ be a pointed metric space and let $E$ be a Banach space. A set $\left((x_i,y_i)\right)_{i\in I}$ of points in $\widetilde{X}$ is called a Lipschitz interpolating set for $\Lip_0(X,E)$ if for each $\left(v_i\right)_{i\in I}\in \ell_\infty(I,E)$, there exists a function $f\in\Lipo(X,E)$ such that 
$$
\frac{f(x_i)-f(y_i)}{d(x_i,y_i)}=v_i
$$
for all $i\in I$. This means that the bounded linear operator $T_E\colon\Lipo(X,E)\to\ell_\infty(I,E)$, defined by 
$$
T_E(f)=\left(\frac{f(x_i)-f(y_i)}{d(x_i,y_i)}\right)_{i\in I}\qquad (f\in\Lipo(X,E)),
$$
is onto. In such a case define the Lipschitz interpolation constant for $(\left((x_i,y_i)\right)_{i\in I},E)$ by
$$
M_E:=\inf\left\{K>0 \,|\, \forall v\in B_{\ell_\infty(I,E)},\; \exists f\in KB_{\Lipo(X,E)}\colon T_E(f)=v\right\}<\infty.
$$
\end{definition}

In the following we will prove that the existence of a Beurling set of functions in $\Lipo(X)$ allows us to connect the Lipschitz interpolating sets for any Banach space $E$.

\begin{proposition}\label{teo-2}
Let $\left((x_i,y_i)\right)_{i\in I}$ be a Lipschitz interpolating set for $\Lipo(X)$ in $\widetilde{X}$ whose Lipschitz interpolating constant is $M$. The following are equivalent:
\begin{enumerate}
\item $\left((x_i,y_i)\right)_{i\in I}$ has a Beurling set $(f_i)_{i\in I}$ of functions in $\Lipo(X)$.
\item For each Banach space $E$, there exists an operator $R_E\in\L(\ell_\infty(I,E),\Lipo(X,E))$ such that $\Vert R_E\Vert\leq M$ and $T_E\circ R_E=\mathrm{Id}_{\ell_\infty(I,E)}$. In such a case, $\left((x_i,y_i)\right)_{i\in I}$ is a Lipschitz interpolating set for $\Lipo(X,E)$ and $M_E=M$. 
\end{enumerate}
\end{proposition}

\begin{proof}
$(i)\Rightarrow(ii)$: Let $E$ be a Banach space. If $(i)$ holds, let $(f_i)_{i\in I}$ be a Beurling set of functions in $\Lipo(X)$ associated to $\left((x_i,y_i)\right)_{i\in I}$ such that, for every $j\in I$, we get
$$
\left(\frac{f_j(x_i)-f_j(y_i)}{d(x_i,y_i)}\right)_{i\in I}=\left(\delta_{ji}\right)_{i\in I}
$$
and 
$$
\sup_{(x,y)\in\widetilde{X}}\sum_{i\in I}\frac{\left|f_i(x)-f_i(y)\right|}{d(x,y)}\leq M.
$$
Define $R_E\colon\ell_\infty(I,E)\to\Lipo(X,E)$ by 
$$
R_E(\alpha)(x):=\sum_{i\in I}f_i(x) \alpha_i\qquad (x\in X,\; \alpha=(\alpha_i)_{i\in I}\in\ell_\infty(I,E)).
$$
An easy calculation shows that $R_E\in\L(\ell_\infty(I,E),\Lipo(X,E))$ with $||R_E||\leq M$ and $T_E\circ R_E=\Id_{\ell_\infty(I,E)}$, and thus $(ii)$ holds. The converse $(ii)\Rightarrow(i)$ is obvious taking $E=\mathbb R$ by Theorem~\ref{rem:levanvarapouloset}.
\end{proof}

Next we prove that, in order to get that a set $\left((x_i,y_i)\right)_{i\in I}$ in $\widetilde{X}$ is a Lipschitz interpolating set for $\Lipo(X,E)$ for every Banach space $E$, the existence of a Beurling set of functions in $\Lipo(X)$ for $\left((x_i,y_i)\right)_{i\in I}$, apart from being sufficient, is also a necessary condition. In order to do so, let us make use of tensor product theory.

Given a pointed metric space $X$, assume that $\left((x_i,y_i)\right)_{i\in I}\subseteq\widetilde{X}$ is a Lipschitz interpolating set for $\Lipo(X,E)$ for every Banach space $E$ with Lipschitz interpolating constant equal to $M_E=M$ (as in the thesis of Proposition~\ref{teo-2}). Given any Banach space $E$, let $T_{E^*}:\Lipo(X,E^*)\to \ell_\infty(I,E^*)$ denote be the Lipschitz interpolating operator associated to $(\left((x_i,y_i)\right)_{i\in I},E^*)$. Consider the identifications 
$$
\Lipo(X,E^*)\cong \mathcal L(\mathcal F(X),E^*)\cong (\mathcal F(X)\pten E)^*.
$$
We claim that the bounded linear operator $R_E:\ell_1(I,E)\to \mathcal F(X)\pten E$, defined by
$$
R_E(v):=\sum_{i\in I} m_{x_i,y_i}\otimes v_i\qquad (v=(v_i)_{i\in I}\in\ell_1(I,E)),
$$
satisfies that $(R_E)^*=T_{E^*}$. Indeed, given $f\in\Lipo(X,E^*)$ and $v=(v_i)_{i\in I}\in\ell_1(I,E)$, we get
\begin{align*}
(R_E)^*(T_f)(v)&=T_f(R_E(v)) =T_f\left( \sum_{i\in I} m_{x_i,y_i}\otimes v_i\right)=\sum_{i\in I}T_f(m_{x_i,y_i})(v_i)\\
&=\underbrace{\left(\frac{f(x_i)-f(y_i)}{d(x_i,y_i)}\right)_{i\in I}}\limits_{\in\ell_\infty(I,E^*)=\ell_1(I,E)^*}\underbrace{\left(v_i\right)_{i\in I}}\limits_{\in\ell_1(I,E)}
=T_{E^*}(f)(v),
\end{align*}
and the arbitrariness of $v$ and $f$ implies $T_{E^*}=(R_E)^*$. By Proposition~\ref{prop:caralipsconstant}, we get that $R_E:\ell_1(I,E)\to\F(X)\pten E$ is an into isomorphism and 
$$
\frac{1}{M}=\max\left\{J>0\colon J\Vert v\Vert\leq\Vert R_E(v)\Vert,\; \forall v\in\ell_1(I,E)\right\}.
$$
Now, \cite[Example 2.6]{ryan} establishes that $Q_E\colon\ell_1(I)\pten E\to\ell_1(I,E)$ given by 
$$
Q_E(\lambda\otimes x):=(\lambda_i x)_{i\in I}\qquad (\lambda=(\lambda_i)_{i\in I}\in\ell_1(I),\; x\in X),
$$
is an onto linear isometry, and this implies that 
$$
\frac{1}{M}=\max\left\{J>0\colon J\Vert z\Vert\leq \Vert (R_E\circ Q_E)(z)\Vert,\; \forall z\in\ell_1(I)\pten E \right\}.
$$
Finally, if we consider the operator $S\in\L(\ell_1(I),\mathcal F(X))$ defined by 
$$
S(\lambda):=\sum_{i\in I}\lambda_i m_{x_i,y_i}\qquad (\lambda=(\lambda_i)_{i\in I}\in\ell_1(I)),
$$ 
then \cite[Proposition 2.3]{ryan} provides the operator $S\otimes_\pi\mathrm{Id}_E\colon\ell_1(I)\pten E\to\F(X)\pten E$ given as 
$$
(S\otimes_\pi\mathrm{Id}_E)(\lambda\otimes x):=S(\lambda)\otimes x\qquad (\lambda\in\ell_1(I),\; x\in E).
$$
It is immediate that $S\otimes_\pi\Id_E=R_E\circ Q_E$. Indeed, given $\lambda\in\ell_1(I)$ and $x\in E$, we get
\begin{align*}
(R\circ Q_E)(\lambda\otimes x)&=R((\lambda_i x)_{i\in I})=\sum_{i\in I} m_{x_i,y_i}\otimes (\lambda_i x)=\sum_{i\in I}\lambda_i m_{x_i,y_i}\otimes x\\
&=\left(\sum_{i\in I}\lambda_i m_{x_i,y_i}\right)\otimes x=S(\lambda)\otimes x=(S\otimes_\pi\Id_E)(\lambda\otimes x).
\end{align*}
It follows that $\ell_1(I)\otimes E\subseteq\ker(S\otimes_\pi\Id_E-R_E\circ Q_E)$, and since the linear span of $\ell_1(I)\otimes E$ is dense in $\ell_1(I,E)\pten E$, we conclude that $S\otimes_\pi\Id_E=R_E\circ Q_E$. All this argument proves the following result.

\begin{proposition}\label{prop:tensorpredual}
Let $X$ be a pointed metric space, let $E$ be a Banach space and let $\left((x_i,y_i)\right)_{i\in I}$ be a set in $\widetilde{X}$. The following are equivalent:
\begin{enumerate}
    \item $\left((x_i,y_i)\right)_{i\in I}$ is a Lipschitz interpolating set for $\Lipo(X,E^*)$ and its Lipschitz interpolating constant is $M_{E^*}$.
    \item The operator $S\otimes_\pi\Id_E:\ell_1(I)\pten E\to\F(X)\pten E$ is an into isomorphism and 
    $$\frac{1}{M}=\max\left\{J>0\colon J\Vert z\Vert\leq \Vert (S\otimes_\pi\Id_{E})(z)\Vert,\; \forall z\in\ell_1(I)\pten E\right\}.$$ $\hfill$ $\Box$
\end{enumerate}
\end{proposition}

Now we are able the present an extension of Proposition~\ref{teo-2} in the following sense.


\begin{theorem}\label{theo:mainvectorval}
Let $\left((x_i,y_i)\right)_{i\in I}$ be a Lipschitz interpolating set in $\widetilde{X}$ whose Lipschitz interpolating constant is $M$. The following are equivalent:
\begin{enumerate}
\item $\left((x_i,y_i)\right)_{i\in I}$ is a Lipschitz interpolating set for $\Lipo(X,E)$ for every Banach space $E$ and the Lipschitz interpolating constant $M_E=M$.
\item $\left((x_i,y_i)\right)_{i\in I}$ has a Beurling set $(f_i)_{i\in I}$ of functions in $\Lipo(X)$.
\end{enumerate}
\end{theorem}

\begin{proof}
$(ii)\Rightarrow(i)$ clearly follows from Proposition \ref{teo-2}. 

Conversely, if we apply $(i)$ to $E=\ell_1(I)$, we get that $T:\Lipo(X,\ell_1(I))\to \ell_\infty(I,\ell_1(I))$ is a Lipschitz interpolating operator associated to $\left((x_i,y_i)\right)_{i\in I}$ and its Lipschitz interpolating constant is $M$. By Proposition~\ref{prop:tensorpredual}, we get that the operator $S\otimes_\pi\mathrm{Id}_{c_0(I)}\colon\ell_1(I)\pten c_0(I)\to\F(X)\pten c_0(I)$ is an into isomorphism and
$$
\frac{1}{M}=\max\{J>0\colon J\Vert z\Vert \leq \Vert (S\otimes_\pi\mathrm{Id}_{c_0(I)})(z)\Vert,\;\forall z\in\ell_1(I)\pten c_0(I)\}.
$$
Set $\varphi:\ell_1(I)\pten c_0(I)\to \mathbb R$ given by the equation
$$
\varphi(\lambda\otimes v):=\sum_{i\in I}\lambda_i v_i\qquad (\lambda=(\lambda_i)_{i\in I}\in\ell_1(I),\; v=(v_i)_{i\in I}\in c_0(I)).
$$
It is a norm-one linear functional (observe $\varphi$ is the identity operator under the identification $(\ell_1(I)\pten c_0(I))^*\cong\mathcal L(\ell_1(I),\ell_1(I))$). Consider $\varphi \circ (S\otimes_\pi  \mathrm{Id}_{c_0(I)})^{-1}: (S\otimes_\pi \mathrm{Id}_{c_0(I)})(\ell_1(I)\pten c_0(I))\to \mathbb R$, which is a functional of norm less or equal than $M$ since $\Vert (S\otimes_\pi \mathrm{Id}_{c_0(I)})^{-1}\Vert=M$. Since $(S\otimes_\pi \mathrm{Id}_{c_0(I)})(\ell_1(I)\pten c_0(I))$ is a subspace of $\mathcal F(X)\pten c_0(I)$, Hahn--Banach theorem provides another continuous linear functional $\Phi\colon\mathcal F(X)\pten c_0(I)\to\mathbb R$ with $\Vert\Phi\Vert\leq M$ and $\Phi(S(\lambda)\otimes v)=\varphi(\lambda\otimes v)$ for all $\lambda\otimes v\in \ell_1(I)\pten c_0(I)$. Let us call $P:\mathcal F(X)\to \ell_1(I)$ the operator defined by
$$
P(\gamma)(v):=\Phi(\gamma\otimes v)\qquad (\gamma\in\F(X),\; v\in c_0(I)).
$$
It is well known (c.f. e.g. \cite[p. 24]{ryan}) that $\Vert P\Vert=\Vert \Phi\Vert\leq M$. 

It remains to prove that $P\circ S=\mathrm{Id}_{\ell_1(I)}$. In order to do so, let $\lambda=(\lambda_i)_{i\in I}\in\ell_1(I)$. Now, given $v=(v_i)_{i\in I}\in c_0(I)$, we infer that 
$$
(P\circ S)(\lambda)(v)=P(S(\lambda))(v)=\Phi(S(\lambda)\otimes v)=\varphi(\lambda\otimes v)=\sum_{i\in I}\lambda_iv_i=\mathrm{Id}_{\ell_1(I)}(\lambda)(v).
$$
The arbitrariness of $v\in c_0(I)$ and $\lambda\in\ell_1(I)$ implies in turn $P\circ S=\mathrm{Id}_{\ell_1(I)}$, and $(ii)$ follows by Theorem \ref{rem-to-cor:caravarapoulos} and Remarks \ref{cor:caravarapoulos} (i).
\end{proof}

\begin{remark}\label{rema:caraprevaravectorvalued}
In view of Theorem \ref{theo:mainvectorval}, Question~\ref{quest:varapoulos} can be equivalently reformulated in the following terms: given a Lipschitz interpolating set $\left((x_i,y_i)\right)_{i\in I}$ in $\widetilde{X}$ for $\Lipo(X)$ whose Lipschitz interpolation constant is $M$, is it true that $\left((x_i,y_i)\right)_{i\in I}$ is a Lipschitz interpolating set for $\Lipo(X,E)$ with Lipschitz interpolation constant $M_E=M$ for every Banach space $E$?
\end{remark}

\section*{Acknowledgments} The authors thank Ram\'on J. Aliaga for fruitful conversations on the topic. Part of the research of the topic was developed during the participation of the second author to the congress “Structures in Banach Spaces” which took place at the Erwin Schr\"odinger International Institute for Mathematics and Physics (ESI) of the University of Vienna, institution which supported part of the expenses. The second author is grateful to the ESI for the financial support received.

The research was supported by Ministerio de Ciencia e Innovaci\'{o}n grant PID2021-122126NB-C31 funded by MICIU/AEI/10.13039/501100011033 and by ERDF/EU. The first author was also supported by Junta de Andaluc\'ia grant FQM194 and by P\_FORT\_GRUPOS\_2023/76, PPIT-UAL, Junta de Andaluc\'ia- ERDF 2021-2027. Programme: 54.A. The second author was also supported by Junta de Andaluc\'ia: grant FQM-0185 and by Fundaci\'on S\'eneca: ACyT Regi\'on de Murcia: grant 21955/PI/22.

\end{document}